\documentclass[12pt]{amsart}
\usepackage{amsmath, amstext, amsthm, amssymb, amsfonts, latexsym}
\usepackage{graphicx} 
\usepackage{enumitem}

\usepackage[colorlinks=true,
            linkcolor=black,
            citecolor=blue,
            urlcolor=black]{hyperref}

\newcommand{\ET}{\mathbf{ET}}
\newcommand{\E}{\mathbb{E}}
\newcommand{\GT}{\mathbf{GT}}

\newcommand{\N}{\mathbb{N}}
\newcommand{\RR}{\mathbb{R}^2}

\newcommand{\Rd}{\mathbb{R}^d}

\newcommand{\R}{\mathbb{R}}
\newcommand{\ST}{\mathbf{ST}}
\newcommand{\TT}{\mathbb{T}^2}
\newcommand{\TTT}{\mathbb{T}^3}
\newcommand{\Td}{\mathbb{T}^d}
\newcommand{\T}{\mathbb{T}}

\newcommand{\Zd}{\mathbb{Z}^d}
\newcommand{\Z}{\mathbb{Z}}
\newcommand{\bL}{\mathbf{L}}
\newcommand{\bP}{\mathbb{P}}

\newcommand{\cA}{\mathcal{A}}
\newcommand{\cE}{\mathcal{E}}
\newcommand{\cF}{\mathcal{F}}
\newcommand{\cL}{\mathcal{L}}
\newcommand{\cM}{\mathcal{M}}
\newcommand{\cR}{\mathcal{R}}

\newcommand{\hA}{\hat{A}}
\newcommand{\hD}{\hat{D}}
\newcommand{\half}{\frac{1}{2}}
\newcommand{\halfc}{\frac{c}{2}}
\newcommand{\hf}{\hat{f}}
\newcommand{\hm}{\hat{m}}
\newcommand{\m}{\mathbf{m}}
\newcommand{\pext}{\pi_{\text{ext}}}

\newcommand{\tU}{\tilde{U}}

\newcommand{\tvarphi}{\tilde{\varphi}}

\newcommand{\ty}{\tilde{y}}

\newcommand{\tu}{\tilde{u}}

\newcommand{\trho}{\tilde{\rho}}
\newcommand{\hx}{\hat{x}}
\newcommand{\hy}{\hat{y}}
\newcommand{\hz}{\hat{z}}
\newcommand{\hU}{\hat{U}}
\newcommand{\zero}{\underline{0}}
\newcommand{\ind}{\operatorname{ind}}
\newcommand{\diam}{\operatorname{diam}}
\newcommand{\trace}{\operatorname{trace}}
\newcommand{\adj}{\operatorname{adj}}

\DeclareMathOperator{\Diff}{Diff}
\DeclareMathOperator{\GL}{GL}
\DeclareMathOperator{\Mtt}{\mathbb{M}_{2 \times 2}(\Z)}
\DeclareMathOperator{\MttR}{\mathbb{M}_{2 \times 2}(\R)}
\DeclareMathOperator{\Mdd}{\mathbb{M}_{d \times d}(\Z)}
\DeclareMathOperator{\MddR}{\mathbb{M}_{d \times d}(\R)}
\DeclareMathOperator{\Mdk}{\mathbb{M}_{d \times k}(\Z)}
\DeclareMathOperator{\Mdo}{\mathbb{M}_{d \times 1}(\Z)}
\DeclareMathOperator{\Ort}{O}
\DeclareMathOperator{\SL}{SL}
\DeclareMathOperator{\SO}{SO}
\DeclareMathOperator{\image}{Im}
\DeclareMathOperator{\rank}{rank}
\DeclareMathOperator{\vspan}{span}

\numberwithin{equation}{section}

\theoremstyle{plain}
\newtheorem{theorem}{Theorem}[section]

\newtheorem{proposition}[theorem]{Proposition}
\newtheorem{corollary}[theorem]{Corollary}
\newtheorem{lemma}[theorem]{Lemma}

\theoremstyle{definition}
\newtheorem{definition}[theorem]{Definition}
\newtheorem{example}[theorem]{Example}

\theoremstyle{remark}
\newtheorem{remark}[theorem]{Remark}

\title[Stably ergodic endomorphisms]{Stably ergodic endomorphisms with positive Lyapunov exponents}

\author{Martin Andersson}

\address{Instituto de Matem\'{a}tica Aplicada. Universidade Federal Fluminense. Rua Professor Marcos Waldemar de Freitas Reis, S/N. 24210-201 Niter\'{o}i, Brazil.}
\email{nilsmartin@id.uff.br}
\urladdr{https://martin-andersson.github.io}

\thanks{The author was partially supported CAPES MATH-AMSUD project "Entropy and
	Dimension in Non-uniformly Hyperbolic Dynamics" and by Projeto
Universal CNPq 401737/2025-0}


\begin{document}

\begin{abstract}
We introduce a new criterion which guarantees that a conservative local
	diffeomorphism is both non-uniformly expanding (in particular has all
	its Lyapunov exponents positive) and stably ergodic. The criterion
	involves  averaging  expansion rates over preimages and allows us to
	construct explicit examples in a variety of homotopy classes,
	including those represented by linear endomorphisms with contracting
	directions. 
\end{abstract}

\maketitle

\section{Introduction}

The present work deals with two related problems concerning conservative
endomorphisms. The first is the flexibility of Lyapunov exponents: 
how close can the exponents be made to one another within a
prescribed homotopy class? In particular, 
under what circumstances can they all be positive? 
The second is the construction and characterization of 
stably ergodic non-uniformly expanding maps.

These two problems are naturally intertwined and lead to a striking contrast
between the infinitesimal dynamics and the large-scale topology. We introduce
a simple criterion for stable ergodicity which produces large families of
conservative endomorphisms homotopic to hyperbolic toral endomorphisms with
contracting directions, but for which all Lyapunov exponents are positive
almost everywhere. Thus the linear model has both expanding and contracting
directions, while the nonlinear conservative representative has only positive
Lyapunov exponents.

As for the first problem, there are two natural levels at which it can be
considered. One may ask only about the integrated Lyapunov exponents, that is,
about positivity and closeness on average with respect to Lebesgue measure.
Alternatively, one may ask for all Lyapunov exponents to be positive almost
everywhere. The first version is rather tractable and admits a simple complete
answer. The second is substantially more difficult and is closely tied to the
stable ergodicity problem above.

Throughout the paper, an endomorphism is a non-invertible $C^1$ local
diffeomorphism $f:M\to M$ of a compact Riemannian manifold. We shall be mainly
concerned with the torus $\Td$, equipped with normalized Lebesgue measure $m$.
An endomorphism is called conservative if it preserves $m$. Equivalently,
\[\sum_{x' \in f^{-1}(x)} \frac{1}{|\det Df(x')|} = 1\]
	for every $x \in M$.

For a conservative endomorphism $f:M \to M$, let
\[
\lambda_1(x,f)\leq\cdots\leq\lambda_d(x,f)
\]
denote its Lyapunov exponents, defined for $m$-almost every $x$, and let
\[
\Lambda_i(f)=\int_{M}\lambda_i(x,f)\,dm(x),
\qquad 1\leq i\leq d,
\]
be the corresponding integrated Lyapunov exponents. Since
\[
\sum_{i=1}^d\Lambda_i(f)
=
\int_{M}\log|\det Df(x)|\,dm(x),
\]
their mean value is constrained by the Jacobian.

When $M = \Td$, $d \geq 2$, if one is only interested in the integrated Lyapunov spectrum, there is
essentially no further obstruction. Indeed, we have the following.

\begin{theorem}\label{thm:integrated_squeeze}
Given any endomorphism $f:\Td\to\Td$, $d\geq2$, and any $\epsilon>0$, there
exists a conservative $C^\omega$ endomorphism $g:\Td\to\Td$ with constant
Jacobian, homotopic to $f$, such that
\[
\Lambda_d(g)-\Lambda_1(g)<\epsilon.
\]
\end{theorem}

\begin{remark}
In fact, Theorem~\ref{thm:integrated_squeeze} also holds for diffeomorphisms,
with no modification to the proof. In the Anosov setting, 
this conclusion is contained in the much more general
flexibility results of Bochi--Katok--Rodriguez Hertz \cite{BKRH22}, who show
that the Lyapunov spectrum can be varied within the corresponding stable and
unstable sign constraints. In particular, in the volume-preserving case the
whole spectrum can be made arbitrarily close to zero. Our interest here is
primarily in the non-invertible setting.
\end{remark}

Every endomorphism of $\Td$ is homotopic to a unique linear endomorphism,
represented by an integer matrix $A$, with
\begin{equation}
|\det A|=\deg(f).
\end{equation}
Thus Theorem~\ref{thm:integrated_squeeze} says that in every non-invertible
homotopy class one can find a constant-Jacobian conservative endomorphism for
which all integrated Lyapunov exponents are arbitrarily close to
\[
\frac1d\log|\det A|>0.
\]
The problem of ergodicity (and hence almost everywhere non-uniform expansion) is
more delicate, and is the main focus of this work.

We say that a conservative local diffeomorphism is \emph{collectively
expanding} if some iterate $F=f^n$ satisfies
\begin{equation} \label{eq:mexpanding}
\sum_{x'\in F^{-1}(x)}
\frac{\log\m(DF(x'))}{|\det DF(x')|}
>0
\end{equation}
for every $x$, where $\m(\cdot)$ denotes the conorm.

\begin{theorem}\label{thm:ergodic_mexpanding}
Every $C^{1+}$ collectively expanding conservative local diffeomorphism is
non-uniformly expanding almost everywhere and stably ergodic.
\end{theorem}

Combining this criterion with the constructions developed later in the paper,
we obtain stably ergodic non-uniformly expanding endomorphisms in certain
hyperbolic homotopy classes. These results give a partial positive answer to
the almost-everywhere version of the problem above.

Our present estimates do not yield arbitrary squeezing of the Lyapunov
spectrum while preserving collective expansion. We do not know of any
obstruction to such a conclusion. In fact, this limitation appears to come
from the construction rather than from the dynamics itself: the collective
expansion estimates used here are carried out at the level of the first
iterate, and allowing suitable higher iterates is likely to provide more
flexibility.

\medskip

The term \emph{non-uniformly expanding} has been used in the literature for
several related, but not equivalent, notions of expansion. In one-dimensional
dynamics, various forms of non-uniform expansion had already appeared in the
study of intermittent and unimodal maps. In higher-dimensional smooth
dynamics, the notion which has perhaps been studied most systematically is the
one developed in the work of Alves and Alves--Bonatti--Viana
\cite{Alves00,ABV}. In the absence of critical points, it is expressed by the
condition
\begin{equation} \label{eq:def_nue}
	\limsup_{n\to\infty}
	\frac1n\sum_{j=0}^{n-1}
	\log\|Df(f^j(x))^{-1}\|<0
\end{equation}
for Lebesgue almost every $x$. In the presence of a critical or singular set,
this is supplemented by a slow-recurrence condition.
In this work an endomorphism is \emph{non-uniformly expanding} if it has an
iterate which satisfies \eqref{eq:def_nue}.

A central tool in this theory is the notion of \emph{hyperbolic times}.
Roughly speaking, non-uniform expansion guarantees the occurrence of
sufficiently many times at which the preceding orbit segment behaves as if it
were uniformly expanding. Such times provide inverse branches with uniform
contraction and bounded distortion. Alves~\cite{Alves00} used this mechanism
to construct absolutely continuous invariant measures for multidimensional
non-uniformly expanding maps, while Alves--Bonatti--Viana~\cite{ABV} developed
closely related ideas in the partially hyperbolic setting with mostly
expanding center direction.

This approach subsequently became the basis for a broader statistical theory
of non-uniformly expanding systems. Alves--Viana~\cite{AV02} established
statistical stability for robust classes of such maps, while
Alves--Luzzatto--Pinheiro~\cite{ALP05} constructed Gibbs--Markov inducing
structures and related the tail of the inducing time to statistical properties
such as decay of correlations and the Central Limit Theorem. Variants of the
same framework have also been developed for systems with critical or singular
sets, where control of recurrence to the critical set is combined with the
hyperbolic-time argument.

Among local diffeomorphisms, the principal robust examples studied in this
context are obtained by deforming uniformly expanding maps; see, for instance,
\cite{Alves00,AV02}. The other paradigmatic family is formed by Viana maps
\cite{Viana97}, where critical points play an essential role. These maps have
subsequently served as a basic model in the study of statistical and stochastic
properties of non-uniformly expanding systems. 

Much of this theory is motivated by the construction and study of absolutely
continuous invariant or physical measures. From this point of view, the
conservative setting is rather different: an absolutely continuous invariant
measure is already given from the outset. To the best of our knowledge,
non-uniform expansion in the sense above has not previously been systematically
studied within the class of conservative endomorphisms. Here the main questions
are instead the occurrence of positive Lyapunov exponents and their relation
with ergodicity and stable ergodicity.

One of the main ideas in this paper is that, in the conservative setting,
averaging expansion over the inverse branches provides a direct mechanism for
producing non-uniform expansion. Indeed, for each $x$ the weights $|\det
Df(x')|^{-1}$ for  $x'\in f^{-1}(x)$, form a probability distribution on the fiber
$f^{-1}(x)$. Collective expansion requires the logarithm of the conorm to have
positive average with respect to this distribution. We show that this fiberwise
condition already implies non-uniform expansion almost everywhere.

Stable ergodicity has been studied predominantly in the setting of
conservative partially hyperbolic diffeomorphisms, beginning with the program
of Pugh and Shub \cite{PS97,PS00} and its connection with accessibility. This
led to a large body of work establishing ergodicity and stable ergodicity under
various accessibility, bunching and center hypotheses; see, among many others,
\cite{BW10,RHRHU08,ACW21}.

By comparison, stable ergodicity for conservative endomorphisms has received
much less attention. A recent exception is \cite{ACS25}, where large open
families of conservative non-uniformly hyperbolic endomorphisms of $\T^2$ were
constructed and, under a natural homotopical condition, shown to be stably
ergodic. Those examples have Lyapunov exponents of opposite signs. The systems
constructed here lie in the complementary regime in which all Lyapunov
exponents are positive.

\subsection*{Use of AI tools}
OpenAI's GPT models 5.6 and 6.0 were used for revision, improving organization, and
smoothing of the text. They also palyed a crucial role in the development of a suitable
example in dimension three (Example~\ref{ex:3D}) and in the proof of
Lemma~\ref{lem:density_of_A_transvections}. The author takes full responsibility
for the mathematical content and the final text.


\section{Transvections and analytic shears}

In this section we recall some facts about transvections and use them to
construct  real-analytic diffeomorphisms of tori, 
homotopic to the identity, whose derivatives are
transvections at every point. Thise "analytic shears" provide the main 
tools in the proofs of  Theorems~\ref{thm:integrated_squeeze} and
\ref{thm:ergodic_mexpanding}.

\subsection{Transvections}

A \emph{transvection} (or \emph{shear map}) on $\R^d$ is an
isomorphism $T:\R^d\to\R^d$ such that $\image(T-I)$ is a
one-dimensional subspace of $\ker(T-I)$. In other words,
\begin{enumerate}
    \item $\rank(T-I)=1$ and
    \item $(T-I)^2=0$.
\end{enumerate}
Equivalently, a transvection can be written as
\[
T=I+uv^T,
\]
where $u,v\in\R^d$ are nonzero orthogonal vectors.
Indeed, $\image(T-I)=\langle u\rangle$ and
$\ker(T-I)=v^\perp$.
Thus $T$ fixes the hyperplane $v^\perp$ pointwise
and shears in the direction of $u$.

In this work, it is convenient to consider the identity as a
transvection, even though it does not satisfy the definition above.

Standard examples of transvections are matrices of the form
\[
T_{(i,j,t)}=I+tE_{i,j}
\]
for $1\leq i,j\leq d$ with $i\neq j$ and $t\in\R$.
We denote the set of all such matrices by $\ET$ and refer
to them as \emph{elementary transvections}.

Let $\GT$ be the set of all transvections on $\R^d$.
We refer to the elements of $\GT$ as \emph{general transvections},
in contrast to the elementary transvections in $\ET$.
It follows immediately from the definition that $\GT$ is closed
under conjugacy by elements of $\GL(d,\R)$.
Conversely, every transvection is conjugate to an elementary one
by an element of $\GL(d,\R)$.

Indeed, given any nontrivial $T\in\GT$, let $M\in\GL(d,\R)$
be such that
\[
\langle Me_1\rangle=\image(T-I),
\qquad
\langle Me_1,\ldots,Me_{d-1}\rangle=\ker(T-I).
\]
Since $(T-I)Me_d$ and $Me_1$ are nonzero vectors in
$\image(T-I)$, there exists $t\in\R$ such that
\[
(T-I)Me_d=tMe_1.
\]
Reorganizing, we get
\[
TMe_d=M(I+tE_{1,d})e_d.
\]
On the other hand, for $1\leq i\leq d-1$,
\[
M(I+tE_{1,d})M^{-1}Me_i
=M(I+tE_{1,d})e_i=Me_i.
\]
Thus
\[
T=M\Sigma_tM^{-1},
\qquad
\Sigma_t=I+tE_{1,d}.
\]

We have therefore established that
\[
\GT=\{MTM^{-1}:T\in\ET,\ M\in\GL(d,\R)\}.
\]
In fact, the conjugating matrix can be chosen orthogonal.
Indeed, using the representation $T=I+uv^T$ with
$u\perp v$, we may choose an orthonormal basis whose
first and last vectors are $u/\|u\|$ and $v/\|v\|$,
respectively. Hence
\[
\GT=\{QTQ^{-1}:T\in\ET,\ Q\in\SO(d,\R)\}.
\]

Let us prove some useful properties of transvections.

\begin{proposition}\label{prop:properties_transvections}
Let $\varphi:\R\to[1,\infty)$ be defined by
\[
\varphi(t)=\|\Sigma_t\|,
\]
and let $\psi:[1,\infty)\to[0,\infty)$ be the inverse
of the restriction of $\varphi$ to $[0,\infty)$.
Then:
\begin{enumerate}
    \item \label{phi_formula}
    For every $t\in\R$,
    \[
    \varphi(t)=\frac{|t|+\sqrt{t^2+4}}{2}.
    \]

    \item \label{phiinverse_formula}
    For every $t\geq1$,
    \[
    \psi(t)=t-\frac1t.
    \]

    \item \label{norm_conorm_difference}
    For every $t\in\R$,
    \[
    \sigma_+(\Sigma_t)-\sigma_-(\Sigma_t)=|t|.
    \]

    \item \label{transvections_orthogonal_change}
    Given transvections $S,T\in\GT$, we have
    $\|S\|=\|T\|$ if and only if there exists
    $Q\in\Ort(d,\R)$ such that $S=QTQ^{-1}$.
    When $d\geq3$, $Q$ may be chosen in $\SO(d,\R)$.

    \item \label{tranvections_parameter_value}
    Let $P\in\SL(d,\R)$ and set
    \[
    \kappa=\|Pe_1\|\cdot\|P^{-T}e_d\|.
    \]
    Then, for every $t\in\R$,
    \[
    \|P\Sigma_tP^{-1}\|=\|\Sigma_{\kappa t}\|.
    \]
\end{enumerate}
\end{proposition}

\begin{proof}
We first prove \eqref{norm_conorm_difference}.
The squares of $\sigma_+(\Sigma_t)$ and
$\sigma_-(\Sigma_t)$ are the eigenvalues of the
nontrivial $2\times2$ block of $\Sigma_t^T\Sigma_t$.
The characteristic polynomial of this block is
\[
\lambda^2-(2+t^2)\lambda+1.
\]
Hence
\[
\begin{aligned}
(\sigma_+(\Sigma_t)-\sigma_-(\Sigma_t))^2
&=\lambda_+-2+\lambda_-\\
&=2+t^2-2=t^2,
\end{aligned}
\]
proving \eqref{norm_conorm_difference}.

For $t\geq0$, we obtain
\[
t=\varphi(t)-\frac1{\varphi(t)}.
\]
This gives \eqref{phiinverse_formula}, and inverting
the formula yields \eqref{phi_formula} for $t\geq0$.
Since $\|\Sigma_{-t}\|=\|\Sigma_t\|$, the formula
holds for all $t\in\R$.

The ``if'' part of \eqref{transvections_orthogonal_change}
is obvious. For the ``only if'' part, first suppose
that $S$ and $T$ are nontrivial.
Let $H_1$ be the span of $e_1$ and $H_{d-1}$ the span
of $e_1,\ldots,e_{d-1}$.
Choose $U,V\in\SO(d,\R)$ such that
\begin{itemize}
    \item $UH_1=\image(S-I)$ and
    $UH_{d-1}=\ker(S-I)$;
    \item $VH_1=\image(T-I)$ and
    $VH_{d-1}=\ker(T-I)$.
\end{itemize}
Then
\[
U^{-1}SU=\Sigma_s,
\qquad
V^{-1}TV=\Sigma_t
\]
for some $s,t\in\R$.
Since $\|\Sigma_s\|=\|\Sigma_t\|$, we have $s=\pm t$.

If $s=t$, then $Q=UV^{-1}$ does the job.
If $s=-t$, change the sign of the first column
of $U$. The resulting orthogonal matrix, still
denoted by $U$, satisfies
\[
U^{-1}SU=\Sigma_t.
\]
When $d\geq3$, we may also change the sign of
one of the intermediate columns, so that
$U\in\SO(d,\R)$. Thus $Q=UV^{-1}$ has the
required properties.

The case where $S$ or $T$ is the identity is immediate,
since equality of their norms then implies that
both are the identity.

Finally, to prove \eqref{tranvections_parameter_value},
write
\[
P\Sigma_tP^{-1}=I+tuv^T,
\qquad
u=Pe_1,\quad v=P^{-T}e_d.
\]
Since
\[
\langle u,v\rangle=e_1^Te_d=0,
\]
we may choose an orthonormal basis whose first and
last vectors are $u/\|u\|$ and $v/\|v\|$,
respectively. The corresponding orthogonal matrix
$Q$ satisfies
\[
Q^{-1}(I+tuv^T)Q=\Sigma_{\kappa t},
\qquad
\kappa=\|u\|\|v\|.
\]
Consequently,
\[
\|P\Sigma_tP^{-1}\|=\|\Sigma_{\kappa t}\|.
\]
\end{proof}

\begin{remark}
The constant $\kappa$ in
Proposition~\ref{prop:properties_transvections}
has a geometric interpretation.
Writing $p_i=Pe_i$ for the columns of $P$, we have
\[
\|P^{-T}e_d\|
=
\|p_1\wedge\cdots\wedge p_{d-1}\|,
\]
where the right-hand side is the $(d-1)$-dimensional
volume of the parallelepiped spanned by the first
$d-1$ columns of $P$.

Indeed, let $q=P^{-T}e_d$. Then
\[
\langle q,p_i\rangle
=\langle e_d,P^{-1}p_i\rangle
=\delta_{d,i}.
\]
Thus $q$ is perpendicular to $p_1,\ldots,p_{d-1}$
and $\langle q,p_d\rangle=1$.

Since $\det P=1$, the volume of the parallelepiped
generated by the columns of $P$ is one.
Computing this volume as base times height gives
\[
1=
\|p_1\wedge\cdots\wedge p_{d-1}\|
\frac{|\langle q,p_d\rangle|}{\|q\|}
=
\frac{\|p_1\wedge\cdots\wedge p_{d-1}\|}{\|q\|}.
\]
Consequently,
\[
\kappa
=
\|p_1\|\,
\|p_1\wedge\cdots\wedge p_{d-1}\|.
\]
In particular, when $d=2$, we have
$\kappa=\|p_1\|^2$.
\end{remark}

We will be particularly interested in transvections
obtained by conjugating elementary transvections
by integer matrices. Let
\[
	\ST=\{PTP^{-1}:T\in\ET,\ P\in\SL(d,\Z)\}.
\]

The relevance of this class is that its elements can
be approximated by derivatives of real-analytic
diffeomorphisms of $\Td$, as we shall see in the
next subsection.

\subsection{Analytic shears}

Let $Z_c:\T \to \R$ ($Z$ for zig-zag) be the function which, upon identifying
$\T$ with $[-\half, \half]$, is defined by
\begin{equation}
	Z_c(x) = \begin{cases}
		 x+\half, \quad & -\half \leq x \leq -\frac{c}{2} \\
		 (1-\frac{1}{c})x, \quad  & |x| < \halfc \\
		 x-\half, \quad  &\halfc \leq x \leq \half
	\end{cases}.
\end{equation}

We wish to smoothen $Z_c$ by a sequence of smooth (analytic) functions
converging to $Z_c$ in $C^0$ and also in $C^1$ outside a neighborhood 
of $\pm c/2$. A convenient way to do this explicitly is by using Ces\`{a}ro averaged Fourier series.
More precisely, denote by
\begin{equation}
	b_n = 2 \int_{\T} Z_c(x) \sin(2 \pi n x) \ dx
\end{equation}
the $n$-th
Fourier coefficient for the (odd) function $Z_c$ and let $s_{c,n}:\T \to \R$
be the map
\begin{equation*}
	s_{c,n} (x) = \frac{1}{n} \sum_{k=1}^{n-1} \sum_{\ell=1}^k 
	b_{\ell} \sin(2 \pi \ell x).
\end{equation*}
This is  in fact the convolution of $Z_c$ with the
Fej\'{e}r kernel
\begin{equation} \label{eq:fejer_kernel}
	F_n(x) = \frac{1}{n} \left( \frac{\sin(n \pi x)}{\sin(\pi x)} \right)^2,
\end{equation}
	which is non-negative.
As a consequence of this, $s_{c,n}$ converges to $Z_c$
uniformly, and $s'_{c,n}$ converges to $Z_c'$ uniformly outside any neighborhood
of $\pm c/2$. More precisely:

\begin{lemma} \label{lem:sprime_convolution} 
	Let $0<c_0 < c< 1$. Then
	\begin{enumerate}[label=(\roman*)]
		\item 	$1-\frac{1}{c_0} \leq s_{c_0,n}'(x) \leq 1$
			\label{est:whole_circle} for every $x \in \T$ and
		\item $s_{c_0,n}'(x) \geq 1-\frac{1}{ n c_0 (c-c_0)}$ 
			for every $x \notin (-\frac{c}{2},\frac{c}{2})$. \label{est:far_from_crit}
	\end{enumerate}
\end{lemma}

\begin{proof}
	Let $\chi$ be the index funcition on $I_{c_0}= (-c_0/2, c_0/2)$. 
	Note that we can write the derivative of $Z_{c_0}(x)$ as
	\begin{equation} \label{Zprime}
		Z_{c_0}'(x) = 1- \frac{1}{c_0} \chi(x) 
	\end{equation}
	so we have 
	\begin{equation} \label{eq:sprime_convolution}
		s_{c_0,n}'(x) = F_n \star Z_{c_0}'(x) = 1 -\frac{1}{c_0} F_n
		\star \chi(x).
	\end{equation}
	Since $F_n$ is nonnegative and has integral one, 
	\[ 
	0 \leq F_n \star \chi (x) \leq 1
	\]
	for every $x \in \T$. That proves \ref{est:whole_circle}.

	Now suppose that $x \notin I_c$.
	Using concavity of $\sin(\pi x)$ on $[0, \half]$ we see that 
	$|\sin(\pi x)| \geq 2|x|$ on $[-\half, \half]$. Combining that with
	\eqref{eq:fejer_kernel} gives the estimate 
	\begin{equation} \label{eq:fejer_estimate}
		F_n(x) \leq \frac{1}{4 n x^2}
	\end{equation}
	in that interval. 
	If $y \in I_{c_0}$, then
	\[
		|x-y|_{\T} \geq \frac{c-c_0}{2}.
	\]
	Consequently,
	\begin{align*}
		F_n \star \chi 
		& = \int_{I_{c_0}} F_n(x-y)\ dy  \\
		& \leq \int_{|t| \geq (c-c_0)/2} F_n(t) \ dt \\
		& \leq 2 \int_{(c-c_0)/2}^{1/2} \frac{ dt}{4nt^2 }
		< \frac{1}{n(c-c_0)}.
	\end{align*}
	Using \eqref{eq:sprime_convolution} we obtain
	\[s_{c_0,n}'(x) \geq 1 - \frac{1}{n c_0 ( c- c_0)}\]
	as required.

\end{proof}

Now let $S_{i,j, c,n,t}: \Td \to \Td$ be the map
\begin{equation*}
	S_{i,j,c,n,t}(x_1, x_2, \ldots, x_d) = (x_1, \ldots, x_i+t s_{c,n}(x_j), \ldots, x_d).
\end{equation*}
Then
\begin{equation} \label{eq:smooth_shears}
	DS_{i,j,c,n,t}(x_1, \ldots, x_d) = I_d + t s_{c,n}'(x_j) E_{i,j}.
\end{equation}
In other words, $DS_{i,j,c,n,t}(x_1, \ldots, x_d)$ is close to $T_{(i,j,t)}$ for
$\frac{c}{2}+\epsilon< x_j < 1-\frac{c}{2}-\epsilon$ and close to 
$T_{(i,j,t(1-\frac{1}{c}))}$
for $-\frac{c}{2}+\epsilon < x_j < \frac{c}{2}-\epsilon$ when $n$ is large.

Let
\begin{equation*}
	\cE = \{S_{i,j, c,n,t} : 1 \leq i , j \leq d,\  
		i \neq j, \ 0<c<1, \ n \in \N, \  t \in \R \},
\end{equation*}
and denote by $\langle \cE \rangle \subset \Diff_m^{\omega}(\Td)$ the group of conservative real
analytic diffeomorphisms of $\Td$ generated by $\cE$.


\section{Squeezing Lyapunov exponents} 
\label{sec:many_transvections}

We can now state a more complete version of
Theorem~\ref{thm:ergodic_mexpanding}.
\begin{theorem} \label{thm:nue}
	Given any $A \in \Mdd$ with $|\det A| \geq 2$, and any
	$\epsilon>0$, there exists $h \in  \langle \cE \rangle $ 
	such that $f = A \circ h$ is non-uniformly expanding on a set of
	Lebesgue measure at least $1-\epsilon$. Moreover,
	\begin{equation}
		\left| \Lambda_i(f) -\log| \det A |^{ \frac{1}{d}}  \right| < \epsilon
	\end{equation}
	for every $1 \leq i \leq d$.

\end{theorem}

Note that Theorem \ref{thm:nue} doesn't say anything about the ergodic
decomposition of Lebesgue measure. However,  we may apply \cite[Corollary
D]{ABV} to conclude that there is a set $H$,
on which $f$ is non-uniformly expanding, is open modulo zero
Lebesgue measure. Moreover, it is a
(possibly countable) union of positive Lebesuge
measure ergodic components, each of which is open  modulo zero. 
In fact, if $d=2$ and $A$ is
either expanding or hyperbolic, then, by the transitivity result in
\cite{And16}, there can be at most one such component.
That is, $f \vert H$ is ergodic.

The proof of Theorem~\ref{thm:nue} is based on the classical fact that any
(orientation-preserving) unimodular matrix is a product of standard
transvections. That
is:
\begin{lemma} \label{lem:product_of_standard_shears}
	Every $B \in \SL(d,\R)$ can be written as a product of 
	 elements of $\ET$. 
 \end{lemma}
 \begin{proof} Gaussian elimination.
 \end{proof}

\begin{remark}
The number of thransvections needed in
Lemma~\ref{lem:product_of_standard_shears} is of order $d^2$. That is not a
problem here, but it is too much when trying to give examples of collectively
expanding maps. This problem will be dealt with in
Section~\ref{sec:fewer_transvections}, where we show that any element of
$\SL(d,\R)$ can be written as the product of an orthogonal matrix and only $d-1$
general transvections.
\end{remark}

\begin{proof}[Proof of Theorem \ref{thm:nue}] 
	Let $\rho = \sqrt[d]{|\det A|}$.
	To prove the Theorem, it sufficies to find $h \in \langle \cE \rangle$
	so that $f = A\circ h$ satisfies 
	\begin{equation}
		\int_{\Td} \log \m(Df(x)) \ dm(x) >  \log \rho -\epsilon^2.
	\end{equation}
	Indeed, in this case the function
	\[
		\Xi(x) = \log \rho  - \log \m(Df(x))
	\]
	is non-negative. This is because $\rho$ is the largest possible value of
	$\m(B)$ for
	any matrix $B$ with $|\det B| = |\det A|$. Let $\Xi^+$ be the
	Birkhoff averate of $\Xi$. Then
	\[\log \rho - \Xi^+(x) \leq \lambda_1(x,f) 
		\leq \ldots \leq \lambda_d(x,f)
	\]
	almost everywhere and Chebyshevs' inequality
	yealds
	\begin{equation}
		\m(\{x: \Xi^+(x) \geq \epsilon \}) 
		\leq \frac{1}{\epsilon} \int_{\Td} \Xi(x) \ d\m(x) 
		< \epsilon.
	\end{equation}
	In particular, the set on which $\Xi^+$ is smaller than $\epsilon$ has
	$m$-measure at least $1-\epsilon$ and, provided that $\epsilon< \log
	\rho$, $f$ is NUE on this set. 
 We have chosen $\rho$ so that $A/\rho \in \GL(d,\R)$. Let us first assume that
 $A/\rho$ is in $\SL(d,\R)$ and comment in the end how to proceed in the general
 case. Under this assumption, there are transvections $T_{(i_k, j_k, t_k)}$, $1\leq k \leq \ell$ such that
 \begin{equation} \label{eq:A_times_elementary_transvections}
	A T_{(i_1, j_1, t_1)} \cdots T_{(i_\ell, j_\ell, t_\ell)} 
	=  \rho I_d.
\end{equation}
	
Let $t_{\max} = \max\{|t_k|: 1 \leq k \leq \ell \}$ and set
\[L_c = \log \m(A) + \ell \log \left( \frac{c}{c+2 t_{max}} \right),\]
and
\[G = \log \rho - \frac{\epsilon^2}{2}.\]
As we will see later, $L_c$ and $G$ will be lower bounds on $\m(Df(x))$ in
"critical" and "non-critical" regions, respectively.
Next, fix $c>0$ 
small enough so that $L_c < G$ and
\begin{equation} \label{eq:def_de_c}
	c \ell L_c + (1-\ell c) G
	> \log \rho - \epsilon^2.
\end{equation}
This can be done because $c \log c$ tends to $0$ as $c \searrow 0$.

Now fix $c_0 = c/2$ and take 
\begin{equation}
	h(x) =  S_{i_1, j_1, c_0, n, t_1} \circ \ldots \circ S_{i_\ell,
j_\ell, c_0, n, t_\ell} (x),
\end{equation}
whith $n$ to be specified later.
Finally, set $f = A \circ h$.

For $1 \leq i \leq d$ write 
$\Delta_c^i = \{(x_1, \ldots, x_d) \in \Td : |x_i|<\frac{c}{2}\}$ (similarly for
$\Delta_{c_0}^i$. 
Let $D_1 = \Delta_{c}^{j_1}$ and, for $2\leq k \leq \ell$, define $D_k$ recursively as
\begin{equation}
	D_k = (S_{(i_{k}, j_{k}, c_0, n, t_{k})})^{-1}D_{k-1} \cup \Delta_c^{j_k}.
\end{equation}
Then $D_\ell$ consists of the points that fall into a critical region at some
point during the shear process. Note that $m(D_\ell) \leq c\ell$.
Note also that $D S_{i_k, j_k, c_0, n, t_k}(x)$ 
converges uniformly to $T_{( i_k, j_k, t_k )}$ outside $\Delta_c^{j_k}$.
Thus we can (and do) choose $n$ so that
	\begin{equation} \label{eq:noncritical_est}
	\log \m(D f(x)) > G
	\end{equation}
	for $x \notin D_\ell$.

For $x \in \Delta_c^{j_k}$ we have the global lower bound
\begin{align} \label{eq:critical_est}
	\m \left( D S_{i_k,j_k, c_0,n,t_k}(x)\right)
	& \geq \frac{1}{\varphi(t_k(1-1/c_0))} \geq
	\frac{1}{\varphi(t_k/c_0)}  \\
	& \geq \frac{1}{1+|t_k|/c_0} \geq
	\frac{c}{c+2 t_{\max}}.
\end{align}
Hence 
\begin{equation*}
	\log \m (D f(x)) \geq \log \m(A) + \ell \log\left(\frac{c}{c+2 t_{\max}}
\right) = L_c
\end{equation*}
for every $x \in D_{\ell}$ (actually, for every $x \in \Td$).

Using \eqref{eq:noncritical_est} and \eqref{eq:critical_est} together with
\eqref{eq:def_de_c} we get the estimate
\begin{align*}
	\int_{\Td} \log \m ( D f(x)) \ dm(x) 
	&= \int_{D_\ell^c} \log \m(D f(x)) \ dm(x) \\
	& +
	\int_{D_\ell} \log \m(D f(x)) \ dm(x) \\
				 & \geq (1-c \ell) \left(\log \rho -
				 \frac{\epsilon^2}{2}\right) 
				 + c \ell L_c  \\
				 &> \log \rho- \epsilon^2
\end{align*}
as required.

It remains to consider the case in which $A/\rho$ is not in $\SL(d,\R)$. In this
case we take any $Q \in \Ort(d) \setminus \SO(d)$ (for example, the matrix swapping
the first and second columns of the identity). Then $QA/\rho \in \SL(d,\R)$, so we
can choose transvections $T_(i_k, j_k, t_k)$, $1 \leq k \leq \ell$ satisfying 
 \begin{equation} 
	A T_{(i_1, j_1, t_1)} \cdots T_{(i_\ell, j_\ell, t_\ell)} 
	=  \rho Q^{-1}
\end{equation}
rather than \eqref{eq:A_times_elementary_transvections}. But since $\m(\rho Q^{-1}) = \m(\rho I) =
\rho$, all estimates on $\m(Df(x))$ remain the same.

\end{proof}

\begin{example} \label{ex:3111}
\begin{figure} 
	\centering
	\includegraphics[width=1.0\textwidth]{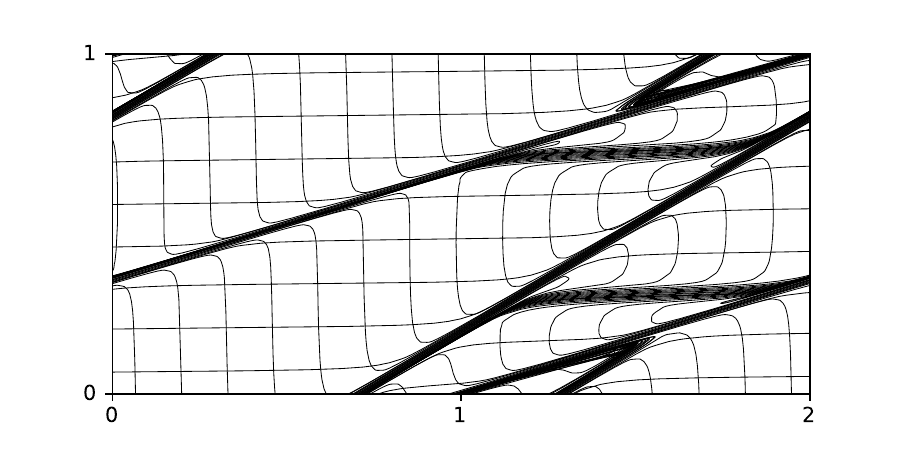}
	\caption{This picture shows the image  of the grid 
	$\R \times \frac{1}{10}\Z \cup \frac{1}{10}\Z \times \R$ 
under the lift of $f_{c,n}$ to $\RR$ with $c=1/10$ and $n = 25$.}
\end{figure}

We now consider an explicit example. 
The map
\[
A=
\begin{pmatrix}
3&1\\
1&1
\end{pmatrix}.
\]
has eigenvalues $2\pm \sqrt{2}$ and is therefeore an Anosov endomorphism.
A direct computation gives
\[
A
\begin{pmatrix}
1&-1+\sqrt{2}\\
0&1
\end{pmatrix}
\begin{pmatrix}
1&0\\
-\frac{\sqrt{2}}{2}&1
\end{pmatrix}
\begin{pmatrix}
1&-3+\sqrt{2}\\
0&1
\end{pmatrix}
=
\sqrt{2}\,I.
\]
Thus, in the notation of the proof of
Theorem~\ref{thm:nue},
\[
	(i_1, j_1) = (1,2), \qquad (i_2, j_2) =
	(2,1), \qquad (i_3, j_3) =
	(1,2),
\]
and
\[
t_1=-1+\sqrt{2},\qquad
t_2=-\frac{\sqrt{2}}{2},\qquad
t_3=-3+\sqrt{2}.
\]
Then it can be read from the proof of Theorem~\ref{thm:nue} that, for $c$ small
enough and sufficiently large $n$, the map
\[
f_{c,n}
=
A\circ S_{1,2,c,n,t_1}\circ S_{2,1,c,n,t_2}\circ S_{1,2,c,n,t_3}
\]
is non-uniformly expanding on a set of Lebesgue measure arbitrarily close to
$1$. Moreover both integrated Lyapunov exponents can be taken as close to $\frac{1}{2} \log 2$
as desired.
 
Let us be more explicit. For that, we need to fix some $0< \beta < \log \sqrt{2}$, 
say $\beta = \frac{1}{3} \log 2$. We then pick some $\beta < G< \log \sqrt{2}$, 
just below $\log \sqrt{2}$, say $G = \frac{9}{20} \log 2$.
We  want explicit values of $c$ and $n$ for which
\[
\Lambda_1(f_{c,n})>\beta.
\]
Since
\[
\m(A)=2-\sqrt{2},
\qquad
t_{\max}=3-\sqrt{2},
\]
the global lower bound used in the proof of Theorem~\ref{thm:integrated_squeeze} is
\[
L_c
=
\log(2-\sqrt{2})
+
3\log\left(
\frac{c}{c+2(3-\sqrt{2})}
\right).
\]

We need $c$ small enogh so that 
\[
3cL_c+(1-3c)G>\frac{1}{3}\log{2}
\]
holds. A possible choice is $c = 10^{-3}$. 

Now take $c_0 = c/2$ and  $\delta = 1/(n c_0(c-c_0)) = 4/(nc^2)$.
Then Lemma~\ref{lem:sprime_convolution} gives
\[
|s'_{c_0,n}-1|\leq\delta
\]
outside $I_c$.

Using the notation from the proof of Theroem~\ref{thm:nue}, $D_3$ is the
critical region, i.e. the set of points where $D f_{c,n}$ is not well
approximated by $\sqrt{2}I$. Fix some $x \notin D_3$ and write 
\[
	B_k = DS_{i_k, j_k, c, n, t_k}(x), \qquad T_k = T_{(i_k, j_k, t_k)}.
\]
Then
\[
	\|B_k - T_k \|\leq
		\frac{4|t_k|}{nc^2}
\]
and
\[ 
	\|B_k \| \leq \| T_k\|
\]
for $1 \leq k \leq 3$.
Hence
\begin{align*}
	Df_{c,n}(x)-\sqrt{2} I
&=
A(B_1-T_1)B_2B_3 \\
&\quad
+AT_1(B_2-T_2)B_3 \\
&\quad
+AT_1T_2(B_3-T_3).
\end{align*}
By the triangle inequality,
\[
	\|Df_{c,n}(x)-\sqrt{2} I\|
\leq C\delta,
\]
where
\[
C=
\|A\|
\left(
|t_1|\|T_2\|\|T_3\|
+
\|T_1\||t_2|\|T_3\|
+
\|T_1\|\|T_2\||t_3|
\right).
\]

For the present values of $t_1, t_2, t_3$ it is possible to verify that
$C<20$. Therefore
\[
	\|Df_{c,n}(x)-\sqrt{2} I\|<20\delta
\]
for every $x \notin D_3$.

Hence, in order for \eqref{eq:noncritical_est} to hold, it is enough to require
\[
	\sqrt{2}-20\delta>e^G.
\]
For $c=10^{-3}$ and $G=\frac9{20}\log2$, the smallest $n$ for which this
inequality holds is $n = 1405591731$. 
In particular $f_{10^{-3},\,2\cdot10^9}$ is an explicit real-analytic 
conservative endomorphism having its smallest integraged Lyepunov exponent in
$(\frac{1}{3}\log{2},
\frac{1}{2}\log{2})$. 
\end{example}


\section{Collectively expanding maps}

This section is dedicated to the proof of Theorem~\ref{thm:ergodic_mexpanding}.
The theorem makes two claims about collectively expanding maps: 
\begin{enumerate}
	\item they are almost everywhere non-uniformly expanding 
		in the sense of \eqref{eq:def_nue}, and
	\item they are ergodic.
\end{enumerate}

The proof of non-uniform expansion is rather straightforward, while ergodicity
is rather more involved.

\subsection{Non-uniform expansion}

Recall that the \emph{transfer operator} $\cL: C^0(\Td, \R) \to C^0(\Td, \R)$ is
given by
\[\cL \varphi (x) = \sum_{x' \in f^{-1}(x)} 
	\frac{\varphi(x')}{|\det Df(x')|}.\]
	In other words, \eqref{eq:mexpanding} is equivalent to say that the transfer
	operator of the function $\xi(x) = \log \m(Df(x))$ is strictly positive.

The proof of almost everywhere non-uniform expansion relies on the well-known
fact that the transfer operator is the dual of the Koopman operator on 
suitable spaces, say $L^2( m)$. In particular
\begin{equation} \label{eq:dual}
	\int (\cL \varphi) \cdot \psi \ dm = \int \varphi \cdot (\psi \circ f)\ dm
\end{equation}
for $\phi, \psi \in C^0(\Td, \R)$.

\begin{proof}[Proof of Theorem \ref{thm:ergodic_mexpanding}: non-uniform expansion] 
	Let
	\[
		\varphi(x) =
		\log \m(DF(x))
	\]
	and denote by 
	\[\tvarphi(x) = \lim_{k \to \infty} \frac{1}{k} \sum_{\ell =0}^{k-1}
	\varphi(F^{\ell}(x)), \]
	the Birkhoff average of $\varphi$ under $F$.
Let 
	\begin{equation} 
		c = \min_{x \in M} \sum_{x' \in F^{-1}(x)} 
		\frac{\varphi(x') }{ |\det DF(x')|}
	\end{equation}

	To prove that $f$ is non-uniformly expanding, it suffices to prove that 
	\begin{equation} \label{eq:tvarphi_at_least_c}
		\frac{1}{m(I)} \int_I \tvarphi \ dm \geq c
	\end{equation}
	for every $F$-invariant Borel measurable set $I \subset M$ of positive
	$m$-measure. That, in fact, implies that $\tvarphi \geq c$ $m$-almost
	everywhere. Indeed, the set where $\tvarphi$ is less than $c$ is
	$F$-invariant. If it were to have positive $m$-measure, it would have to violate
	\eqref{eq:tvarphi_at_least_c}.

	To this end, let $I$ be any  $F$-invariant Borel set, and denote by $\chi_I$ the index
	function at $I$. Invoking \eqref{eq:dual} with $\chi_I$ in place of
	$\psi$, we get
\begin{align}
	\int_I \tvarphi \ dm 
	& = \int_I \varphi \ dm 
	 = \int \varphi \cdot \chi_I \ dm \\
	& = \int \varphi \cdot \chi_I \circ F \ dm  
	 = \int \cL \varphi \cdot \chi_I \ dm  \\
	& = \int_I \left( \sum_{x' \in F^{-1}(x)} 
	\frac{\varphi(x')}{|\det DF(x')|} \right) \ dm \geq c \cdot m(I)
\end{align}
as required. That proves that $f$ is non-uniformly expanding.
\end{proof}

\subsection{The supermartingale property of backward dynamics}

Let $\bL_f$ be the inverse limit space of $f:M \to M$ and  $\hm$ the lift of
the Lebesgue measure to $\bL_f$. For $x \in M$, let $\Sigma_x =
\pext^{-1}(x)$ be the fiber over $x$ in $\bL_f$. Then $\Sigma_x$ is a cantor
set. For every $y \in f^{-n}(x)$ we define the $n$-cylinder $[y]_n = \{\hy \in
	\Sigma_x : y_n = y\}$. The collection $\{[y]_n : n \in \N \text{ and }y
	\in f^{-1}(x)\}$ generates the topology on $\Sigma_x$. Let 
$m_x$ be the unique measure on $\Sigma_x$ satisfying
\begin{equation*}
	m_x([y]_n) = |\det Df^n(y)|^{-1}
\end{equation*}
for every $n \in \N$ and $y \in f^{-n}(x)$.

Let $\cA_n$ be the $\sigma$-algebra on $\Sigma_x$ generated by $n$-cylinders and let
\begin{align*}
	X_n(\hx) &= \sum_{k=1}^{n} \log \| Df (x_k)^{-1}\| \\
		     &= - \sum_{k=1}^{n} \log \m (Df (x_k)).
\end{align*}
By convention we set $X_0 = 0$.

Clearly $X_n$ is $\cA_n$-measurable, as it depends only on the first $n$ entries
of $\hx$. 

We have
\begin{align*}
	\E (X_{n+1} \vert \cA_n ) (\hx)
	& = \sum_{y \in f^{-1}(x_n)} 
	\frac{\left( X_{n}(\hx) -\log \m(Df(y)) \right)}{| \det Df(y)|} \\
	& = X_n(\hx) - \sum_{y \in f^{-1}(x_n)} 
	\frac{\log \m(Df(y))}{|\det Df(y)|} \\
	& < X_n(\hx).
\end{align*}

In other words, $X_{n+1}$ is a \emph{supermartingale}. More precisely, if we
denote by $c$ the infimum of the left hand side of \eqref{eq:mexpanding}, then
\begin{equation}
	\E(X_{n+1} \vert \cA_n) \leq X_n-c.
\end{equation}

\begin{theorem}[Azuma's inequality] \label{teo:azuma}
	Let $(M_n, \cA_n)$ be a supermartingale on a probability space $(\Omega, \mathbb{P})$. Suppose that there exists $L$ such that 
	\[|M_{n+1}-M_n| \leq L \text{ almost surely.}\]
	Then, for every $\lambda>0$ and every $n \geq 0$ we have
	\begin{equation} 
		\mathbb{P}(M_n \geq M_0+ \lambda) \leq \exp \left(-\frac{\lambda^2}{2 L^2 n} \right).
	\end{equation}
\end{theorem}

\begin{corollary} \label{cor:Azuma}
	Given any $\epsilon>0$ and $\alpha>0$ there exists $C>0$ such that
	\[\mathbb{P}(\{M_n-M_0 \leq C+\alpha n \ \forall n \geq 0 \})>1-\epsilon.\]
\end{corollary}

\begin{proof} 
	Taking $\lambda = C + \alpha n$ in Theorem \ref{teo:azuma}
	gives us the bound
	\begin{equation}
		\mathbb{P}(M_n-M_0 \geq C+\alpha n) 
		< \exp\left(-\frac{C \alpha}{L^2}\right) \cdot \exp\left(-\frac{\alpha^2 n}{2 L^2}\right).
	\end{equation}
	Hence
	\begin{align*}
		& \mathbb{P}\left(\bigcup_{n\geq 1} 
		\{M_n - M_0 \geq C+\alpha n\}\right) \\
		& < \exp\left(-\frac{C \alpha}{L^2}\right) \cdot
		\frac{1}{1-\exp\left(-\frac{\alpha^2}{2 L^2}\right)} < \epsilon
	\end{align*}
	for sufficiently large $C$.
	\end{proof}

\subsection{Local unstable manifolds}
	A connected open set $U \subset \Td$ is called \emph{inessential} if
	every closed curve in $U$ is nullhomotopic in $\Td$. In other words, $U$
	is inessential if the  homomorphism $i_\star: \pi_1(U) \to \pi_1(\Td)$ is
	trivial, where $i:U \to \Td$ is the inclusion map.
	It is traightforward to verify that if $U \subset \Td$ is inessential,
	then $f^{-n}(U)$ is the union of $d^n$ connected components, all of
	which are inessential. Given $U \subset \Td$ inessential and $\hx =
	(x_0, x_1, \ldots) \in \bL_f$ with $x_0 \in U$, denote by
	\begin{itemize}
		\item $U_{\hx}^n$ the
	connected component of $f^{-n}(U)$ that contains $x_n$; by
	\item $\hU_{\hx}^n = \hf^n(\pi^{-1}(U_{\hx}^n))$ the set of points
		$(y_0, y_1, \ldots)
		\in \pi^{-1}(U)$ for which $y_i \in U_{\hx}^i$ for every $0 \leq
		i \leq n$; and by
	\item $\hU_{\hx} = \bigcap_{n \geq 1} \hU_{\hx}^n$
	the set of points $\hy = (y_0, y_1, \ldots) \in \bL_f$ for which $y_n
	\in U_{\hx}^n$ for every $n \geq 0$.
	\end{itemize}

	Now, if we write $\Sigma = \pi^{-1}(x_0)$ we may identify $U\times
	\Sigma$ with an open subset of $\bL_f$ through the mapping
	\[\phi_U(y,\hx) = (y_0, y_1, y_2, \ldots),\]
	where $y_n$ is the unique point in $U_{\hx}^n$ which is mapped to $y$ by
	$f^n$. 

	Although it is true that the volume of $U_{\hx}^n$ goes to zero (ineed exponentially fast), these sets are, in general, not contained in small balls. On the contrary, if we consider a sequence of lifts $\tU_{\hx}^n$, their diameter may very well go to infinity exponentially fast. 

	Still, when $f$ is non-uniformly expanding, there are many choices of
	$U$ and $\hx$ for whch $\diam(U_{\hx}^n) \to 0$. In this case, we refer
	to $\hU_{\hx}$ as an \emph{local unstalbe manifold} of $\hx$ under
	$\hf$. (Note that the local unstalbe manifold is a subset of $\bL_f$,
	rather than $\Td$.) The following is a reformulation of \cite[Lemma 2.7]{ABV}. 

\begin{lemma} \label{lem:unstable_mfd_abv}
	Given $0< \sigma < 1$ there exists $0< \delta < \half$ such that if
	for some $\hx \in \bL_f$ we have
	\begin{equation}
		\prod_{k=1}^{n} \| Df(x_k)^{-1}\| \leq \sigma^n
	\end{equation}
	for every $n \geq 1$, then, taking $U = B_\delta(x_0)$, we have
	\begin{equation}
		d(x_n,y_n) \leq \delta \sigma^{n/2}
	\end{equation}
	for every $(y_0, y_1, \ldots) \in \hU_{\hx}$ and $n \geq 0$.
\end{lemma}
In particular, Lemma \ref{lem:unstable_mfd_abv} says that $\diam(U_{\hx}^n) \to
0$.

Combining Lemma \ref{lem:unstable_mfd_abv} and corollary \ref{cor:Azuma} we get:

\begin{lemma} \label{lem:lots_of_unst_mfds}
	Let $f:M \to M$ be as in the hypothesis of 
	Theorem \ref{thm:ergodic_mexpanding}.
	For every $\epsilon>0$ there exists $\delta> 0$ such that
	$m_x(L_\delta)>1-\epsilon$ for every $x \in M$.
\end{lemma}

\begin{proof}
	Fix $c>0$ so that the left hand side of \eqref{eq:mexpanding} is greater
	than $c$ for every $x$. Apply Corollary \ref{cor:Azuma} with $M_0 = 0$,
	$M_n = X_n + nc/2$ for $n \geq 0$, and $\alpha = c/4$. Then $M_n-M_0
	\leq C + \alpha n$ reads as $X_n \leq C - nc/4$, so Corollary
	\ref{cor:Azuma} says that the set of points for which
	\[ \prod_{k=1}^n \| Df(x_k)^{-1}\| < e^{C-nc/4}\]
	for every $n \geq 1$ has $m_x$-measure greater than $1-\epsilon$.

	For every $j\ge1$ we have
	\[
	\prod_{r=1}^j
	\|D(f^m)(x_{rm})^{-1}\|
	\le
	\prod_{k=1}^{mj}\|Df(x_k)^{-1}\|
	\le
	e^{C-cmj/4}
	\le
	\sigma^j.
	\]
	Thus $\hx$ satisfies the hypothesis of Lemma~\ref{lem:unstable_mfd_abv}
	for the map $f^m$. Consequently, there exists $\delta_0>0$, independent
	of $x$ and of $\hx$, such that the corresponding inverse branch for
	$f^m$ contracts a ball of radius $\delta_0$ exponentially.

	To obtain the same conclusion for all intermediate iterates of $f$, let
	\[
	K_m=
	\max_{0\le r<m}\sup_{z\in M}\|Df^r(z)\|.
	\]
	Taking
	\[
	\delta=\frac{\delta_0}{K_m},
	\]
	we obtain a local unstable manifold of size $\delta$ for $f$ itself.

	Therefore, for every $x\in M$, the set $L_\delta\subset\Sigma_x$ of
	points admitting a local unstable manifold of size $\delta$ satisfies
	\[
	m_x(L_\delta)>1-\epsilon.
	\]
\end{proof}

\subsection{Ergodicity}
Since collective expansion is preserved under iteration,
and ergodicity of an iterate implies ergodicity of the
original map, we may replace $f$ by a sufficiently large
iterate. In view of the proof of
Lemma~\ref{lem:lots_of_unst_mfds}, we may therefore assume
that, for some $0<\sigma<1$, the set
\[
H_x=\left\{\hat y\in\Sigma_x:
\prod_{k=1}^n\|Df(y_k)^{-1}\|\leq\sigma^n
\text{ for every }n\geq1\right\}
\]
has positive $m_x$-measure for every $x$.

Fix an arbitrary point $x\in\Td$, $\delta>0$ as in
Lemma~\ref{lem:lots_of_unst_mfds}, and let
\begin{itemize}
	\item $ D = B_\delta (x)$,
	\item $\Sigma = \pi^{-1}(x)$,
	\item $H = \{\hy \in \Sigma: \prod_{k=1}^n \| Df(y_k)^{-1}\|
		\leq \sigma^n \ \forall n \geq 1$ \},
	\item $\Omega = D \times H$, identified with its image under $\phi_D$,
	\item $m_D$ be the restriction of $m$ to $D$,
	\item $\hm_{\Omega}$ be the restriction of $\hm$ to $\Omega$ (seen through
		$\phi_D$),
	\item $\cF_{\ty}^n:D \to D_{\ty}^n$, $\ty \in \Sigma$ be the inverse of $f^n \vert
		D_{\ty}^n : D_{\ty}^n \to D$,
	\item $\nu$ be the measure on $\Sigma$ defined on cyliners
		\[
			[y_1, y_2,
			\ldots, y_n] = \{\hz \in \Sigma_x: z_i = y_i
			\text{ for } 1 \leq i \leq n\}
		\]
		by
		\[ \nu([y_1, \ldots, y_n]) =  m(D_{\hy}^n).\]
		Equivalently, $\nu = \int_D m_z \ dm(z)$. 
\end{itemize}

Clearly, $D, \Sigma, H, \Omega, m_D, \hm_\Omega$ and $\nu$ all depend on $x$ but
this dependence is omitted for neatness.
	
\begin{lemma} \label{lem:measure_m_hat_omega}
	The measure $\hm_{\Omega}$ is equivalent to $m_D \times \nu_{H}$.
	Moreover, the Radon Nikodym derivative of the former with respect to the
	latter is bounded from above and from zero.
\end{lemma}

\begin{proof}
Let $\cF_n$ be the filtration of $\sigma$-algebras on $D \times \Sigma $ generated
by sets of the form $D \times [a_1, \ldots, a_n]$. 
Define
\[
\rho_n:D\times\Sigma\longrightarrow (0,\infty)
\]
by
\begin{equation}
\label{eq:rho_n}
\rho_n(z,\hy)
=
\frac{|\det D\cF_{\hy}^n(z)|}
{\nu([y_1, \ldots, y_n])}.
\end{equation}
In particular, $\rho_n$ is constant in the second coordinate on every
$n$-cylinder, and hence is $\cF_n$-measurable. 

By construction, $\rho_n$ is the Radon--Nikodym derivative of the restriction
of $\hm$ to $\cF_n$ with respect to the restriction of $m_D\times\nu$ to
$\cF_n$. Consequently, $(\rho_n,\cF_n)_{n\geq0}$ is a martingale.

For $n\geq1$, let
\[
H_n=
\left\{
\hy \in \Sigma:
\prod_{k=1}^j\|Df(y_k)^{-1}\|
\leq\sigma^j
\text{ for every }1\leq j\leq n
\right\}.
\]
Then $H_n\in\cF_n$,
\[
H_{n+1}\subset H_n,
\qquad
H=\bigcap_{n\geq1}H_n.
\]
Let
\[
\tau(\hx)=\inf\{n\geq1:\hx\notin H_n\},
\]
with the convention $\tau(\hx)=\infty$ if $\hx\in H$.
Thus $\tau$ is a stopping time.

Consider the stopped martingale
\[
\trho_n=\rho_{n\wedge\tau}.
\]
By \cite[Proposition 2.8]{ABV}, there exists $K\geq1$ such that
\[
\frac1K
\leq
\frac{|\det D\cF_{\hy}^n(z)|}
     {|\det D\cF_{\hy}^n(z')|}
\leq K
\]
for every $z,z' \in D$ and every $n$ for which $\hy\in H_n$.
At the stopping time, the same estimate holds after increasing $K$ if
necessary, since only the last inverse iterate is not covered by the
hyperbolic estimate and its distortion is uniformly bounded by compactness
and the $C^{1+\alpha}$ regularity of $f$.

It follows from \eqref{eq:rho_n} that
\[
\frac{1}{K\,m(D)}
\leq
\trho_n(y,\hx)
\leq
\frac{K}{m(D)}
\]
for every $n$ and almost every $(y,\hx)\in D\times\Sigma$.
Hence $(\trho_n,\cF_n)$ is a uniformly integrable martingale. By the martingale
convergence theorem, $\trho_n\longrightarrow\trho$ almost surely and in $L^1(m_D\times\nu)$.

If $\hx\in H$, then $\tau(\hx)=\infty$, so $\trho_n(y,\hx)=\rho_n(y,\hx)$ for every $n$. 
Therefore $\rho_n$ converges on $\Omega=D\times H$ to $\rho=\trho|_{\Omega}$, and
\[
\frac{1}{K\,m(D)}
\leq
\rho
\leq
\frac{K}{m(D)}
\]
almost everywhere on $\Omega$.

Since $\rho_n$ is the finite-level density of $\hm$ with respect to
$m_D\times\nu$, passing to the limit on $\Omega$ shows that
\[
\frac{d\hm_{\Omega}}
     {d(m_D\times\nu_H)}
=
\rho.
\]
In particular, $\hm_{\Omega}\sim m_D\times\nu_H,$ and the Radon-Nikodym derivative is bounded from above and away from zero.
\end{proof}

We are now in position to conclude ergodicity of colectively expanding maps. The
proof is a variation of the famous Hopf argument \cite{Hopf}. 

\begin{proof}[Proof of Theorem~\ref{thm:ergodic_mexpanding}: Ergodicity]

	We fix $f: M \to M$ as in the hypothesis of
	Theorem~\ref{thm:ergodic_mexpanding}. Let $\delta>0$ be as in 
	Lemma~\ref{lem:lots_of_unst_mfds}. Fix an arbitrary $x \in \Td$ and let
	$D = B_\delta(x)$ and take $\Sigma$, $H$, and $\Omega$ as above. Let
	$\cR$ be the set of points in $\bL_f$ for which forward and backward
	Birkhoff averages exist and coincide under all continuous functions. Our
	aim is to show that the empirical measures map $e:\cR \to
	\cM_{\hf}(\bL_f)$, given by
	\[e_{\pm}(\hx) = \lim_{n \to  \infty} \frac{1}{n} \sum_{k=0}^{n-1}
	\delta_{\hf^{\pm k} (\hx)},\]
	is constant $\hm$-almost everywhere on $\pi^{-1}(D)$. Then ergodicity
	follows by covering $\Td$ by disks of this form and observing that
	Birkhoff averages must coincide almost everywhere on the overlaps.

	To see why $e$ is constant on $\pi^{-1}(D)$, recall from
	Lemma~\ref{lem:measure_m_hat_omega} that $\hm_{\Omega}$ is equivalent to
	$m_D \times \nu_H$. Consequently, there is
	some $\hy \in H$ such that $\pi(\hD_{\hy}\cap \cR)$ has full $m$-measure
	in $D$. Then $e_-$ is well-defined and constant everywhere on
	$\hD_{\hy}$. Hence $e_{+}$ is well-defined and constant on a subset $D'
	\subset \hD_{\hy}$ such that $\pi(D')$ has full $m$-measure in $D$. But
	$e_+$ is constant on fibers of $\pi$. Hence $e_+$ is constant on $\cR' =
	\pi^{-1}(D')$, which has full $\hm$-measure in $\pi^{-1}(D)$. 

\end{proof}

\subsection{Dissipative perturbations}

The arguments in this work rely heavily on the assumption that $f$ preserves
Lebesgue measure. In particular, the definition of collective
expansion and the proof of Theorem~\ref{thm:ergodic_mexpanding}
make essential use of how the lift of Lebesgue measure to the natural extension
disintegrates along fibers.

Nevertheless, there is an interesting connection with collective expansion 
and the notion of virtually expanding
maps introduced by Tsujii in \cite{Tsu23}.

Given a smooth local diffeomorphism
$f:M\to M$ and $\mu>0$, let
\[
B_\mu(f)=
\sup_{\substack{x\in M\\0\neq\xi\in T_x^*M}}
\sum_{y\in f^{-1}(x)}
\frac{1}{|\det Df(y)|}
\left(
\frac{\|\xi\|}{\|Df(y)^*\xi\|}
\right)^\mu.
\]
The map $f$ is called virtually expanding if there exist
$\mu>0$ and $n\geq1$ such that
\[
B_\mu(f^n)<1.
\]
Tsujii proved that virtual expansion is a $C^1$-open
condition in the space of smooth local diffeomorphisms,
and that the associated Perron--Frobenius operator is
quasi-compact on a Sobolev space of positive order.
Using this, he showed that every virtually expanding
map admits finitely many ergodic absolutely continuous
invariant measures, whose basins cover Lebesgue almost
every point.
We shall show that every conservative collectively
expanding map is virtually expanding. (Thank you Sylvain Crovisier, 
Ali Tahzibi and Meysam Nassiri for pointing  me in this direction.)
The converse, however, is not true. Collective expansion is stronger than
virtual expansion, since it averages the minimal expansion rates of the inverse
branches, whereas virtual expansion takes into account the different directions
of expansion along those branches.

\begin{proposition}\label{prop:CE_implies_VE}
Every $C^\infty$ conservative collectively expanding
local diffeomorphism is virtually expanding.
\end{proposition}

\begin{proof}
Let $F=f^n$ be an iterate satisfying
\eqref{eq:mexpanding}. For $\mu\geq0$, define
\[
G_\mu(x)=
\sum_{y\in F^{-1}(x)}
\frac{\m(DF(y))^{-\mu}}{|\det DF(y)|}.
\]
Since $F$ preserves Lebesgue measure, we have
$G_0(x)=1$ for every $x$. Moreover,
\[
\left.\frac{\partial G_\mu(x)}{\partial\mu}
\right|_{\mu=0}
=
-\sum_{y\in F^{-1}(x)}
\frac{\log\m(DF(y))}{|\det DF(y)|}<0.
\]
By compactness, this inequality is uniform in $x$.
Consequently, for all sufficiently small $\mu>0$,
\[
\sup_{x\in M}G_\mu(x)<1.
\]

Recall that Tsujii's quantity $B_\mu(F)$ is
defined by taking the supremum, over nonzero
cotangent vectors, of the weighted sum of their
inverse expansion rates over preimages.
Since these inverse expansion rates are bounded
above by $\m(DF(y))^{-1}$, we have
\[
B_\mu(F)\leq\sup_{x\in M}G_\mu(x)<1.
\]
Thus $B_\mu(f^n)<1$, which implies that $f$
is virtually expanding.
\end{proof}

\begin{corollary}
Let $f$ be a $C^\infty$ conservative collectively
expanding endomorphism. Then there exists a
$C^1$ neighborhood $\mathcal U$ of $f$ such that
every $C^\infty$ endomorphism $g\in\mathcal U$
admits finitely many ergodic absolutely continuous
invariant measures, whose basins cover Lebesgue
almost every point.
\end{corollary}


\section{Constructing collectively expanding maps}
\label{sec:fewer_transvections}

We now turn to the construction of collectively expanding maps in different
homotopy classes.

\subsection{A matrix decomposition}

Now that we have established non-uniform expansion and ergodicity of
collectively expanding maps, a natural question arises: How common are they?
Does every homotopy class contain one? Although we do not give a complete answer
to this question, we will show that many homotopy classes do, including
hyperbolic ones.

To appreciate what makes the question hard, consider the map in Example
\ref{ex:3111}. Is it collectively expanding? Certainly, it cannot satisfy
\eqref{eq:mexpanding} with $F = f^1$. That is because each point has only two
preimages, and if one falls into the critical region, its contraction cannot be
compensated by the expansion in the other preimage.
However, it is a priori possible
that a higher iterate of $f^n$ will work. If for every $x$, a sufficently large
share of the $2^n$ points in $f^{-n}(x)$ fall into the good region (where the
derivative is close to $\sqrt{2} I$), this may very well compensate for the few
preimages that fall into the bad (critical) region. But it is very hard to keep
track of all preimages of $f^n$, even for $n=1$. Fortunately, there is a much
more economical way of deforming a linear map on $\TT$ so that its derivative at
every point is close, not to a dilatation, but a dilatation times an orthogonal
map. For examle, the map $A$ from Example~\ref{ex:3111} can be written as the
product $\sqrt{2}QT$ where
\[
	Q = \frac{1}{\sqrt{2}}
	\left(
		\begin{matrix}
			1 & 1 \\
			-1 & 1
		\end{matrix}
	\right)
\]
is orthogonal and
\[
	T = 	
	\left(
		\begin{matrix}
			1 & 0 \\
			2 & 1
		\end{matrix}
	\right)
\]
is a transvection. In fact this can always be done, although it is not always
possible to choose the transvection in $\ST$ (see Example~\ref{ex:5111}).

\begin{lemma} \label{lem:orthogonal_times_shear_dim2}
	Every $A \in \SL(2,\R)$ can be written as $A=QT$, where $Q\in \SO(2,\R)$ and $T \in \GT$.
\end{lemma}

\begin{proof}
	Since $\det A = \sigma_+(A) \cdot \sigma_-(A)= 1$, there exists a unit
	vector $v$ such that $\|Av\| = 1$. Let $U, V \in \SO(2,\R)$ be such that
	$Ve_1 = v$ and $Ue_1 = Av$. Then, writing $S = U^{-1} A V$ we have $S e_1 = e_1$.
	Moreover, $\det S = \det V^{-1} \det A \det U = 1$. Hence 
	\[S = \left(
			\begin{matrix}
				1 & t \\
				0 & 1
			\end{matrix}
	\right)\] for some $t \in \R$. Now, $A = U S V^{-1} = (UV^{-1})(V S
	V^{-1})$, so the proof follows by taking $Q = U V^{-1}$ and $T = V S
	V^{-1}$.
\end{proof}

The point of Lemma \ref{lem:orthogonal_times_shear_dim2} is that it allows us to
consider deformations of $A$ by considering maps of the form
\begin{equation} \label{eq:f_me_general_2d}
	f = A \circ P \circ S_{c,n,t} \circ P^{-1},
\end{equation}
where
\begin{equation}
	S_{c,n,t}(x,y) = (x+ts_{c,n}(y), y)
\end{equation}
is shorthand for the map $S_{1,2,c,n,t}$ from section
\ref{sec:many_transvections}, and $P\in \SL(2,\Z)$ and $t$ are chosen so 
that $ P \Sigma_tP^{-1}$ is either equal to $T$ (if $T$ has a rational shear
direction), or close to $T$ (if it has not). 
The advantage of considering maps of this form is that it is much easier to
control how many pre-images fall into a "bad" region (where the conorm of the
derivative is far away from $\sqrt{|\det A|}$).

Lemma~\ref{lem:orthogonal_times_shear_dim2} has a natural generalization in
higher dimensions. 

\begin{lemma} \label{lem:orthogonal_times_shears_dimn}
	Every $A \in \SL(d,\R)$ can be written as $A = Q T_{d-1} \cdots T_{1}$,
	where $Q \in \SO(d,\R)$ and $T_i \in \GT$ for $i = 1, \ldots, d-1$.
	Moreover, the transvections $T_i$ can be chosen so that
	\begin{equation*}
		\image(T_{j}-I) \subset \ker(T_{i}-I)
	\end{equation*}
	for every $1\leq j < i \leq d-1$.
\end{lemma}

The proof uses the following notational convention: If $M$ is a $d \times d$
matrix, then $\hat{M}$ is the $(d+1) \times (d+1)$ matrix given by 
\begin{equation*}
	\hat{M} = \left(
	\begin{array}{c|c}
		M & \begin{matrix} 0 \\ \vdots \\ 0 \end{matrix} \\ \hline
		0 \cdots 0 & 1
	\end{array}
\right).
\end{equation*}

\begin{proof}[Proof of lemma \ref{lem:orthogonal_times_shears_dimn}]
	We argue by induction on $d$. Lemma \ref{lem:orthogonal_times_shear_dim2}
	says that the conclusion is true for $d=2$. Now suppose it is true for
	$d$ and let us deduce that it is then true for $d+1$.

	Since $\det A = 1$, there exists some subspace $L \subset \R^{d+1}$ of
	dimension $d$ such that the Jacobian of $A \vert_L : L \to A(L)$ 
	is equal to $1$. Let $H = \vspan( e_1, e_2, \ldots, e_{d})$ and
	take $U, V \in \SO(d+1,\R)$ be such that $VH = L$ and $UH = AL$. Then
	$C = U^{-1}AV$ leaves $H$ invariant. Upon
	possibly inverting the sign of the last and one more column of U, we may
	(and do) assume that $C \vert_H$ is orientation preserving. 
	Since the Jacobian of
	$A \vert_L$ is $1$, so is the Jacobian of $C \vert_H$. 
	Hence $C$ must be of the form
		\begin{equation*}
			\left( 
				\begin{array}{c|c}
				B & 
				\begin{matrix}
					v_1 \\ \vdots \\ v_{d}
				\end{matrix} \\ \hline
				\begin{matrix} 0 & \cdots & 0 \end{matrix} & 1
		\end{array}
	\right)
		\end{equation*}
		where $B \in \SL(d,\R)$ and $v = (v_1, \ldots,
		v_{d})^{T} \in \R^{d}$.

Let $W \in \SO(d,\R)$ be such that $Wv = t e_1$ for some $t\in \R$. Then
\begin{equation*}
	\hat{W} C = \left(
	\begin{array}{c|c}
		WB & \begin{matrix} t \\ 0 \\ \vdots \\ 0 \end{matrix} \\ \hline
		\begin{matrix} 0 & \cdots & 0  \end{matrix} & 1
	\end{array}
\right)
\end{equation*}

The matrix $WB$ is an element of $\SL(d,\R)$ so, by our inductive hypothesis,
it can be written as $P  S_{d-1} \cdots S_1$ where $P \in \SO(d,\R)$ and $S_i
\in \GT$
for $1 \leq i \leq d-1$.  Thus 
\begin{equation*}
	\hat{W} C = \left(
		\begin{array}{c|c}
			P \cdot  S_{d-1} \cdot  \ldots \cdot  S_1 & \begin{matrix} t \\ 0 \\ \vdots \\ 0
			\end{matrix} \\ \hline
				\begin{matrix} 0 & \cdots & 0 \end{matrix} & 1
			\end{array}
		\right)
		= \Sigma_t  \cdot \hat{P} \cdot \hat{S}_{d-1} \cdot  \ldots \cdot  \hat{S}_1,
	\end{equation*}
	where $\Sigma_t = I_{d+1}+t E_{1, d+1}$. Hence $A$ can be written as
	\begin{equation*}
		A  = U C V^{-1} =  U \hat{W}^{-1}\Sigma_t \hat{P}\hat{S}_{d-1} \cdots \hat{S}_1 V^{-1} 
		= Q T_d \cdots T_1,
	\end{equation*}
	where
	\begin{align*}
		Q & =  U \hat{W}^{-1} \hat{P} V^{-1}, \\
		T_d &= V \hat{P}^{-1} \Sigma_t \hat{P} V^{-1}, \text{ and} \\ 
		T_i &= V \hat{S}_i V^{-1} \text{ for } 1 \leq i \leq d-1.
	\end{align*}

	Suppose now that the $S_j$ were chosen so that
	\[
		\image(S_{j}-I) \subset \ker(S_i-I)
	\]
	for every $1 \leq j< i \leq d-1$. Then the same holds for
	$\hat{S}_i$ and, consequently, for $T_i=V\hat{S}_iV^{-1}$. So, to complete
	the proof, it suffices to prove that
	\[
		\image(T_{d-1}-I) \subset
		\ker(T_d-I).
	\]
	To that end, observe that $\hat{P}H=H$, so
	\[
		\ker(T_d-I) = V\hat{P}^{-1}\ker(\Sigma_t - I) \supset
		V\hat{P}^{-1}H = VH = L.
	\]
	On the other hand, $\image(\hat{S}_{d-1}-I) \subset H$, so
	\[
		\image(T_{d-1}-I) =
		V\image(\hat{S}_{d-1}-I) \subset VH = L
	\]
	as required.
\end{proof}

\subsection{Bounding the number of preimages in the critical strip}

In order for a map $f$ of the form in \eqref{eq:f_me_general_2d} (or its higher
dimensional analogue) to be colectively expanding, we want
$t$ not to be too big, $c$ not to be too small, and also that the number of
preimages of any point by $f$ is as small as possible. We shall see that this
can be controlles by analyzing the arithmetic properties of the shear plane
relative to $A$.

\begin{definition}
	Let $M \in \Mdk$ be a rank $k$ matrix. We define the \emph{index} of $\ind(M)$ of $M$ to be the product of its invariant factors. Equivalently, $\ind(M)$ is the greatest common divisor of all $k \times k$ minors of $M$.
A matrix vector of index one is called primitive.
\end{definition}

Here, a vector $v \in \Z^d$ is identified with an element of $\Mdo$, so 
\[
\ind(v) = \max \{k \in \N: v = k w \text{ for some } w \in \Zd\},
\]
or, equivalently, the greatest common divisor of its components. A primitive vector is a vector which is \emph{visible from the origin}, i.e. the line segment from the vector to the origin in $\Rd$ does not contain other points in $\Zd$ except the origin and itself. More generally, a matrix $M \in \Mdk$ has index $n$ means that the image of $[0,1)^k$ under $M$ intersects $\Zd$ in $n$ points.  

Let $E$ be the $d \times (d-1)$ matrix with ones in the diagonal and zeros
elsewhere. In other words, $E$ acts on $d \times d $ matrices on the right by
deleting the last column.

\begin{lemma} \label{lem:preimages_in_critical_stripe_dimd}
	Let
	\begin{equation}
		f = A \circ P \circ S_{c,t} \circ P^{-1}
	\end{equation}
	for some $P \in \GL(n, \Z)$ and some $c \leq 1/\tau_d$. Then, given any $x
	\in \Td$,  the cardinality of $f^{-1}(x)\cap P\Delta_c$ is at most $\ind(APE)$.

\end{lemma}

\begin{proof}
	Since $P$ maps $\Delta_c$ bijectively onto $P\Delta_c$ and since
	$S_{c,t}$ leaves $\Delta_c$ invariant, the cardinality of $f^{-1}(x)
	\cap P\Delta_c$ equals the cardinality of $(AP)^{-1}(x)\cap \Delta_c$.
	Note that $\tau_d (AP)^{-1}$ is an integer matrix. In particular,
	$(AP)^{-1}(\zero)$ is contained in $\tau_d^{-1}\Zd$. Therefore
	$(AP)^{-1}(x)$ is a t translation of a subset of $\tau_d^{-1}\Zd$. But
	the height of $\Delta_c$ is $2c<\tau_d$, so the cardinality of
	$(AP)^{-1}(x)$ is either zero or equal to the number
	\[
		k := \# (AP)^{-1}(\zero) \cap \T^{d-1}\times \{0\}.
	\]
	 
	Now let us consider $A$ and $P$ as maps on $\Rd$ rather than $\Td$. Then
	$k$ is equal to the index of $AP(\Z^{d-1}\times \{0\})$ as a subgroup of
	$\Zd \cap P(\R^{d-1}\times \{0\})$. That is,
	\[k = [\Zd \cap APE(\R^{d-1}): APE(\Z^{d-1})].\]

	Let $LDR$ be the Smith decomposition of $APE$. 
	Then $L$ maps
	$\Z^{d-1}\times\{0\}$ bijectively onto $\Zd \cap APE(\Rd)$. Hence
	\[k=[\Zd \cap DR(\R^{d-1}): DR(\Z^{d-1})].\]
	But $DR(R^{d-1}) = \R^{d-1} \times \{0\}$ and $R(\Z^{d-1}) =
	R(\Z^{d-1})$, so
	\[k = [\Z^d \times \{0\}: D(\Z^{d-1})] = \ind(D) = \ind(APE).\]
\end{proof}

\begin{remark} \label{rmk:alternative_calc_of_ind}
The index in Lemma~\ref{lem:preimages_in_critical_stripe_dimd}
can be computed directly from the normal vector to
the invariant hyperplane, without explicitly
constructing the matrix $P$.

Indeed, let $v=P^{-T}e_d$, so that $v$ is a primitive
integer normal vector to
$P(\R^{d-1}\times\{0\})$. Then
\[
\ind(APE)=\gcd\bigl(\adj(A)^T v\bigr),
\]
where the gcd is taken over the coordinates of the
vector and $\adj(A)=(\det A)A^{-1}$.

To see this, observe that the coordinates of
$\adj(A)^T v$ are, up to signs, the maximal minors
of $APE$, whose gcd is precisely $\ind(APE)$.
In particular, if $S=I+tuw^T,$ where $u,w\in\Z^d$, $\langle u, w \rangle =0$, and $w$ is primitive,
then the invariant hyperplane of $S$ is $w^\perp$,
and the corresponding index is 
$\gcd\bigl(\adj(A)^T w\bigr).$ 
\end{remark}

\begin{lemma} \label{lem:primitive_density}
	Let  \[ V = \{ v \in \Zd: \ind(Av) = \tau_1 \}. \]
	Denote by $p:\Rd \to \bP(\Rd)$ the canonical projection.
	Then $p(V)$ is dense in $\bP(\Rd)$.
\end{lemma}

\begin{proof}
Using the Smith Normal form decomposition $A = LDR$ and conjugating with $L$ if necessary, it is enough to consider the case in which $A$ is diagonal with invariant factors $\tau_1, \ldots, \tau_d$ on the diagonal.
By further dividing by $\tau_1$ if neccessary, it is enough to consider the case in which $\tau_1 = 1$. 

Consider some non-zero $x \in \Rd$ and choose a primitive 
\[
	w = (w_1, \ldots, w_d)^T \in \Zd
\]
such that $p(w)$ is arbitrarily close to $p(x)$ and $w_1$ is i
relatively prime to $\tau_d$. We claim that $Aw$ is primitive.
Indeed, suppose there is some prime $p$ dividing every component of $Aw$. 
But
\[
	Aw = (w_1, \tau_2 w_2, \ldots, \tau_d w_d)
\]
so, in particular, $p$ divides $w_1$. 
Now, since $w_1$ and $\tau_d$ are coprime, $p$ does not divide $\tau_d$. Hence $p$ does not divide any of the $\tau_i$. But $p$ divides all components of $Aw$, so $p$ must divide each of the $w_i$. But that contradicts the assumption of $w$ being primitive. 
\end{proof}

\subsection{Density of good shear planes}
We shall now see that, given any linear endomorphism $A$ with invariant factors
$(\tau_1, \ldots, \tau_d)$, not only is it is possible to approximate any 
transvection in $\GT$ whth one in $\ST$, but the choice can be made so that we
minimize the index in the hyperplane fixed by the shear.
More precisely, let 
\[
	\ST_A = \{ P \Sigma_t P^{-1}: P \in \SL(d,\Z), \ t \in \R, \ind(APE) = \tau_1
	\cdots \tau_{d-1}\}.
\]

\begin{lemma} \label{lem:density_of_A_transvections}
	The set $\ST_A$ is dense in $\GT$.
\end{lemma}

\begin{proof}

Fix some $T \in \GT$ and $\epsilon> 0$. 
Write $T = Q \Sigma_t Q^{-1}$ for some $Q \in \SO(d, \R)$. 
Denote by $q_i = Q e_i$ the columns of $Q$. 
Let $B = \det A (A^{-1})^T$. Then $B$ has invariant factors
\[\sigma_1 | \ldots | \sigma_d \]
where $\sigma_i = |\det A|/\tau_{d+1-i}$. Hence, by Lemma~\ref{lem:primitive_density}, 
we can find a primitive $u \in \Zd$
with 
\[\left\| \frac{u}{\|u\|} - q_d \right\| < \epsilon,\]
such that $\ind(Bu) = \sigma_1 =  \tau_1 \cdots
\tau_{d-1}$. Since $u$ is primitive, we can find an element $R \in SL(d, \Z)$
having $u^T$ as its last row. That is, $R^T e_d = u$.

Let $K = R^{-1}$. Then 
\[u^T KE = e_d^T RKE = e_d^T E = (0, \ldots, 0).\]
That is, $u$ is perpendicular to all but the last columns of $K$, or, stated
differently, $KE$ is a primitive rank $d-1$ matrix whose columns span the
orthogonal complement of $\vspan(u)$. 

Let $\tu = u/\|u\|$ and denote by $r = q_1 - \langle q_1,\tu\rangle \tu$ the orthogonal
projection of $q_1$ to $u^\perp = \image(KE)$. 
Then
\[
	\|r-q_1\| = |\langle q_1, \tu \rangle| = |\langle q_1, \tu-q_d \rangle| \leq
	\|\tu-q_d\| < \epsilon.
\]
We now consider the cases $d=2$ and $d \geq 3$ separately.

{\bf Case $d=2$:}
In this case, $r = \lambda K e_1$ for some $\lambda$. Setting $s = t
\lambda/\|u\|$, $S = K \Sigma_s K^{-1}$ we get
\[S = I + t\lambda K e_1 \tu^T = I + t r \tu^T, \]
so 
\[S-T = t \left[ (r-q_1) \tu^T + q_1 (\tu-q_d)^T \right].\]
Recall that the norm of a rank one matrix $x y^T$ is just the product of the
norms of $x$ and $y$. Hence
\[\|S-T\| \leq |t| \left(\|r-q_1\| + \| \tu-q_d \|\right) < 2|t| \epsilon .\]

{\bf Case $d \geq 3$:}
Choose a primitive vector $v \in \Z^{d-1}$ so that
\[\|\lambda KEv -r \| < \epsilon\]
for some $\lambda \in \R$.  Combining this with the previous estimate and
writing $w = KEv$ we get
\[
\| \lambda w - q_1 \| < 2 \epsilon.
\]
Choose $M \in \SL(d-1,\Z)$ having $v$ as its first
column and take $P = K \hat{M}$. 
Then $PE = KEM$, so in particular $Pe_1 = w$  and 
$(P^{-1})^T e_d = (K^{-1})^Te_d = u$. 
Taking
$s=t\lambda/\|u\|$ and $S = P \Sigma_s P^{-1}$ we get
\[S = I + t \lambda w \tu^T. \]
Hence
\begin{align}
	S-T & = t\lambda w \tu^T - t q_1 q_d^T \\
	    & = t \left[ (\lambda w \tu^T- q_1 \tu^T) + (q_1 \tu^T-q_1 q_d^T)
	    \right],
\end{align}
and
\[\|S-T\| \leq |t| \left( \|\lambda w - q_1\| + \|\tu-q_d\| \right) <
3|t|\epsilon.\]
Finally, note that 
\[
	\ind(APE) = \ind(AKEM) =
	\ind(AKE) = \ind(Bu) = \tau_1 \cdots \tau_{d-1},
\]
so $S \in \ST_A$. As $\epsilon$ was chosen arbitrarily, this proves the 
density of $\ST_A$ in $\GT$.


\end{proof}

\subsection{Examples of collectively expanding maps}

We now illustrate how to use the tecniques developed in Section
\ref{sec:fewer_transvections} to produce explicit examples of collectively
expanding maps. 

Recall that the Frobenius norm of a matrix $A \in \MddR$ is given by
\[\|A \|_{F} = \sqrt{\trace(A^T A)}.\]
Recall also that the trace of $A^T A$ is the sum of the squares of the singular
values of $A$. In particular, for $A \in \MttR$, we have
\begin{equation} \label{eq:frobenius_2d}
\|A\|_F^2 = \|A\|^2 + \m(A)^2.
\end{equation}

\begin{example} \label{ex:5111}

	Consider the matrix
	\[A = \left(\begin{matrix} 5 & 1 \\ 1 & 1 \end{matrix} \right). \]
	It has determinant $4$ and eigenvalues $3 \pm \sqrt{5}$, so it induces an Anosov endomorphism on $\TT$.
	We can write $A$ as $2 Q S$ with
	\[Q = \left(\begin{matrix}\frac{2}{3} & \frac{\sqrt{5}}{3}\\-
	\frac{\sqrt{5}}{3} & \frac{2}{3}\end{matrix}\right)\]
	and $S = U \Sigma_t U^{-1}$, where $t = -\sqrt{5}$ and
			\[U = \frac{1}{\sqrt{150-12\sqrt{5}}} \left(\begin{matrix}-3 + 2 \sqrt{5} & -11\\11 & -3 + 2
	\sqrt{5}\end{matrix}\right)\]
	The first column of $U$ points in a direction which can be fairly 
	well approximated by $(2,15)$. (Using continud fractions, 
	one finds that $\frac{15}{2}$ is the second convergent of
	$U_{21}/U_{11}$.) 
	The pair $(2,15)$ has B\'{e}zout coefficients $-7$ and $1$, so
					\[P =\left( \begin{matrix} 2 & -1 \\ 15
					& -7 \end{matrix} \right)\]
	belongs to $\SL(2,\Z)$.
	The number $t/(2^2+15^2)$ (see Proposition 
	\ref{prop:properties_transvections}, (5)) is 
	well approximated by $t'=-1/102$.
	Let $T = P \Sigma_{t'} P^{-1}$. Then
	\[
		H := A T^{-1} = A P \Sigma_{-t'}P^{-1} = 
		\frac{1}{102} \left(\begin{matrix}135 & 152\\-153 & 136\end{matrix}\right)
	\]
	should not be far from $2$ times an
	orthogonal matrix. In fact we claim that
	\begin{equation} \label{eq:good_region_estimate_5111}
		\frac{99}{50} < \m(H) < \|H \| < \frac{101}{50}.
	\end{equation}
	Indeed, using \eqref{eq:frobenius_2d} together with th identity
	$\m(H) \cdot \|H \| = |\det (H)| = 4$ and direct computation of the Frobenius
	norm of $H$, we get
	\[
		\|H\|^2 + \frac{16}{\|H\|^2} = \|H\|_F^2 = 8 + \frac{1}{5202}.
	\]
	The function $\zeta(x) = x^2 + 16/x^2$ is strictly increasing on $(2, \infty)$
	and satisfies
	\[\zeta\left(\frac{101}{50}\right) > 8 + \frac{1}{5202}.\]
	Hence $\|H\|^2 < (101/50)^2$. The lower bound for $\m(H)$ in
	\eqref{eq:good_region_estimate_5111} then follows
	from the identity $\m(H)\cdot \|H\| = 4$.

	Now fix $c_0 = 3/14$ and define 
	\[f_n = A P S_{c_0,n,-t'}P^{-1}.\]
\begin{figure} 
	\centering
	\includegraphics[width=0.9\textwidth]{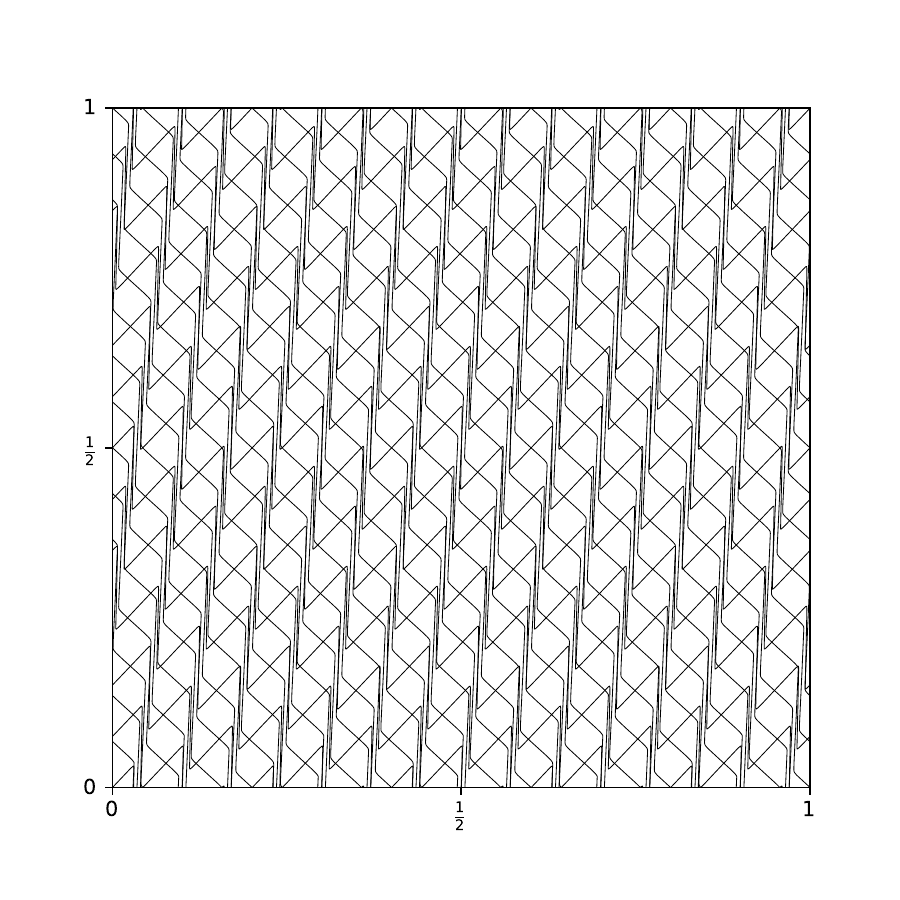}
	\caption{This picture shows the preimage of the grid 
	$\R \times \frac{1}{10}\Z \cup \frac{1}{10}\Z \times \R$ 
under the lift of $f_{n}$ with $c=1/10$ and $n = 50$.}
	\label{fig:3shears}
\end{figure}
	At any point $x \in \TT$, the derivative of $f_n$ is of the form
	\begin{equation} \label{eq:form_of_derivative}
	AP\Sigma_{-r}P^{-1} = H P\Sigma_{t'-r}P^{-1}
\end{equation}
	for some $r$ in the image of $t's_{n,c_0}$. 
	By Lemma~\ref{lem:sprime_convolution},
	\[
		r=t's'_{c_0,n}(x)\in
		\left[-\frac1{102},\frac{11}{3\cdot102}\right],
	\]
	so that
	\[
		t'-r\in\left[-\frac7{153},0\right].
	\]

	Let $\kappa = \|Pe_1\|^2 = 229$. By
	Proposition~\ref{prop:properties_transvections}, the norm of $M_{t'-r} =
	P \Sigma_{t'-r}P^{-1}$ satifies
	\[\|M_{t'-r}\| - \frac{1}{\|M_{t'-r}\|} = 229|t'-r| .\]
	We have seen that $|t'-r| \leq 7/153$.  It is also straightforward to
	check that
	\[ \frac{229 \cdot 7}{153} < 11 - \frac{1}{11}.\]
	Hence $\|M_{t'-r}\|<11$, so that
	\begin{equation} \label{eq:global_lower_conorm_bound_5111}
		\m(Df_n(x) ) \geq \m(H) \m(P\Sigma_{t'-r}P^{-1}) > \frac{99}{50}
	\cdot \frac{1}{11} = \frac{9}{50}
\end{equation}
	for every $x \in \TT$.

The vector $APe_1 = (25,17)^T$ is primitive, so by
Lemma~\ref{lem:preimages_in_critical_stripe_dimd}, any point
$x$ in $\TT$ has at most one pre-image under $f_n$ in $P\Delta_{c}$.

In order for $f_n$ to be collectively expanding, we need
\begin{equation} \label{eq:product_version_of_CE}
	\prod_{x' \in f_n^{-1}(x)} \m(Df_n (x')) >1
\end{equation}
for every $x \in \TT$. Since there is at most one $x'$ in $P\Delta_{c}$, it is
enough to choose $n$ sufficiently big so that $\m(Df(x'))$ is greater than, say,
$9/5$. For then, 
\eqref{eq:global_lower_conorm_bound_5111} gives that the left hand side of
\eqref{eq:product_version_of_CE} is greater than
\begin{equation} \label{eq:product_of_derivatives}
	\frac{9}{50} \cdot
	\left(\frac{9}{5}\right)^3,
\end{equation}
which is greater than $1$.

Now, the derivative of $f_n$ is of the form
\eqref{eq:form_of_derivative}, and
\[
	\begin{aligned}
		\m(HP\Sigma_{t'-r}P^{-1})
		&\geq \m(H)\m(P\Sigma_{t'-r}P^{-1})\\
		&>\frac{99}{50}\m(P\Sigma_{t'-r}P^{-1}).
	\end{aligned}
\]
Thus it suffices to take $|t'-r|$ small enough
so that
\[
	\m(P\Sigma_{t'-r}P^{-1})>\frac{10}{11}.
\]

We have already seen that
\[
	\frac{1}{\m(M_{t'-r})}-\m(M_{t'-r})=229|t'-r|.
\]
Thus, we need
\[
	229|t'-r|<
	\frac{11}{10}-\frac{10}{11},
\]
or
\[
	|t'-r|<\frac{21}{110\cdot229}.
\]
Choosing
\begin{equation}\label{eq:size_of_r}
	|t'-r|<\frac1{1200}
\end{equation}
will do the job.

Now, since $r=t's'_{c_0,n}(x)$ and $t'=-1/102$,
condition \eqref{eq:size_of_r} is equivalent to
\[
	\frac1{102}(1-s'_{c_0,n}(x))<\frac1{1200}.
\]
Taking $c_0=3/14$ and $c=5/21$ in
Lemma~\ref{lem:sprime_convolution}, we find
that $f_n$ is collectively expanding for every
$n\geq2306$.
Moreover, since $f_n$ has constant Jacobian $4$,
its smallest Lyapunov exponent satisfies
\[
\lambda_1(x,f_n)\geq
\frac14\log\left(\frac{85}{81}\right)>0
\]
for Lebesgue almost every $x$.
\end{example}

\begin{example} \label{ex:3D}
Consider the integer matrix
\[
A=
\begin{pmatrix}
4&-7&-8\\
1&2&-8\\
-1&1&5
\end{pmatrix}.
\]
It has determinant $27$, invariant factors
$(\tau_1,\tau_2,\tau_3)=(1,3,9)$, and characteristic
polynomial
\begin{equation}
p_A(t)=t^3-11t^2+45t-27.
\end{equation}
In particular, $A$ is hyperbolic, with one real and two complex eigenvalues
satisfying
\[0< \lambda_1 < 1 < |\lambda_2| = |\lambda_3|.\]
A direct calculation shows that $A$ can be written as
$A=3QT_1T_2$, with
\[
T_1=
\begin{pmatrix}
0&1&2\\
1&0&-2\\
-1&1&3
\end{pmatrix},
\quad
T_2=
\begin{pmatrix}
0&1&0\\
-1&2&0\\
0&0&1
\end{pmatrix},
\quad
Q=\frac13
\begin{pmatrix}
-2&-1&-2\\
1&2&-2\\
2&-2&-1
\end{pmatrix}.
\]
Here $T_1,T_2\in\ST$ and $Q\in\SO(3,\R)$.
Moreover,
\[
\image(T_2-I)\subset\ker(T_1-I).
\]

The invariant planes of $T_1$ and $T_2$ have primitive
normal vectors
\[
a_1=(1,-1,-2),\qquad a_2=(1,-1,0),
\]
respectively. We see that
\[
\adj(A)^Ta_1=(9,9,18),
\qquad
\adj(A)^Ta_2=(15,15,48),
\]
which are vectors of index $9$ and $3$, respectively
(see Remark~\ref{rmk:alternative_calc_of_ind}).
Recall from
Lemma~\ref{lem:preimages_in_critical_stripe_dimd}
that these numbers give upper bounds for the number
of preimages in the critical regions, where
$\mathfrak m(Df)$ may be small. In order to reduce
this number to the optimal value $\tau_1\tau_2=3$,
we replace $T_1$ by a nearby transvection whose
invariant plane has index $3$.

The shear directions of $T_1$ and $T_2$ are
\[
b_1=(1,-1,1),
\qquad
b_2=(1,1,0),
\]
respectively. Our replacement for $T_1$ should
therefore have a shear direction close to $b_1$,
while its invariant plane must still contain $b_2$.

Take
\[
u=(21,-20,20),
\qquad
v=-u\times b_2=(20,-20,-41),
\]
and define
\[
T_1'=I-\frac1{400}uv^T.
\]
Since $\langle u,v\rangle=0$, this is a transvection.
Moreover, $\langle v,b_2\rangle=0$, so that
\[
\image(T_2-I)\subset\ker(T_1'-I).
\]
The directions of $u$ and $v$ are
close to the original shear direction
$b_1$ and normal vector $a_1$, respectively.

Furthermore, $v$ is primitive and
\[
\adj(A)^Tv=(177,177,345).
\]
Since $\gcd(177,345)=3$, the invariant plane of
$T_1'$ has index $3$
(see Remark~\ref{rmk:alternative_calc_of_ind}).

Now fix
\[
\frac1{10}=c_0<c<\frac19.
\]
Choose $P_1,P_2\in\SL(3,\Z)$ such that
\[
(T_1')^{-1}
=P_1\Sigma_{1/400}P_1^{-1},
\qquad
T_2^{-1}=P_2\Sigma_1P_2^{-1},
\]
and write
\[
F_1=P_1S_{1,3,c_0,n,1/400}P_1^{-1},
\qquad
F_2=P_2S_{1,3,c_0,n,1}P_2^{-1}
\]
(see \eqref{eq:smooth_shears}). Take
\[
f_n=A\circ F_2\circ F_1.
\]
Thus $f_n$ is conservative, real analytic, and has
constant Jacobian $27$.

Let
\[
\Delta_c=
\{(x_1,x_2,x_3)\in\T^3:|x_3|<c/2\},
\]
where the third coordinate is represented in
$(-1/2,1/2]$. The strip $P_1\Delta_c$ is foliated by
planes parallel to $\ker(T_1'-I)$, and is therefore invariant
under $F_1$. Similarly, $P_2\Delta_c$ is invariant
under $F_2$. But recall that
\[
\image(T_2-I)\subset\ker(T_1'-I),
\]
so $F_2$ leaves $P_1\Delta_c$ invariant as well.

Write
\[
B_1=P_1\Delta_c,
\qquad
B_2=F_1^{-1}(P_2\Delta_c).
\]
Since $(\tau_1,\tau_2,\tau_3)=(1,3,9)$ and
$c<1/9$,
Lemma~\ref{lem:preimages_in_critical_stripe_dimd}
implies that
\[
\#\bigl(f_n^{-1}(x)\cap B_i\bigr)\leq3
\]
for $i=1,2$ and every $x\in\T^3$.

Consequently, for every $x \in \TTT$, $f_n^{-1}(x)$ has at most 6 points in the
"bad" region $B_1 \cup B_2$. 

We will now show that  $\m(D f_n (x))$ is close to $3$ outside $B_1 \cup
B_2$,  establish a global lower bound for $\m(Df_n(x))$, and verify that
\eqref{eq:mexpanding} with $F=f_n$.

Recall that we have chosen $c_0=1/10$.
Hence, from Lemma~\ref{lem:sprime_convolution},
the derivative of $s_{c_0,n}$ is bounded below
by $-9$ and above by $1$.

Thus, writing
\[
N_1=(T_1')^{-1}-I=\frac1{400}uv^T,
\qquad
N_2=T_2^{-1}-I=b_2a_2^T,
\]
we see that the derivative of $f_n$ at any
point $x\in\T^3$ can be written in the form
\[
Df_n(x)=A(I+aN_2)(I+bN_1)
\]
for some $-9\leq a,b\leq1$ (depending on $x$).

Since $(T_1')^{-1}$ and $T_2^{-1}$ are
transvections, we have $N_1^2=N_2^2=0$.
Moreover, from
\[
\image(T_2^{-1}-I)
\subset\ker((T_1')^{-1}-I)
\]
we also have $N_1N_2=0$. Hence
\[
\begin{aligned}
(Df_n(x))^{-1}
 \bigl(AT_2^{-1}(T_1')^{-1}\bigr)
&=(I-bN_1)(I-aN_2)(I+N_2)(I+N_1)\\
&=I+(1-b)N_1+(1-a)N_2(I+N_1).
\end{aligned}
\]

Put
\[
H=AT_2^{-1}(T_1')^{-1}.
\]
A direct calculation gives
\[
H=\frac1{400}
\begin{pmatrix}
-780&-420&-781\\
480&720&-904\\
780&-780&-419
\end{pmatrix}.
\]
The matrix $H^TH-\frac{196}{25}I$ is
positive definite. Hence
\[
	\langle H v, Hv \rangle - \frac{196}{25} \langle v, v
	\rangle>0,
\]
for every non-zero $v \in \R^3$, so
\[
\m(H)>\frac{14}{5}.
\]

Since $D f_n$ converges uniformly to $H$ outside $B_1 \cup B_2$, for large $n$
we have 
\begin{equation} \label{eq:lower_bound_mDf_3D}
	\m(Df_n(x))> \frac{14}{5}
\end{equation}
for every $x \notin B_1 \cup B_2$.

It remains to establish a uniform lower bound for $\m(Df_n(x))$.
Using the Frobenius norm, a direct
computation gives
\[
\max_{\alpha,\beta\in\{0,10\}}
\left\|
\bigl[I+\alpha N_2(I+N_1)+\beta N_1\bigr]
H^{-1}
\right\|_F^2 <32^2.
\]
By convexity, the same upper bound holds for
every $\alpha,\beta\in[0,10]$. Consequently,
\[
\|(Df_n(x))^{-1}\|<32
\]
for every $x$, or, equivalently,
\[
\m(Df_n(x))>\frac1{32}
\]
for every $x \in \TTT$.

Finally, we conclude by noticing that for such $n$,
\[
	\prod_{x' \in f_n^{-1}(x)} \m(D f_n(x')) >
	\left(\frac{14}{5} \right)^{21} 
	\left( \frac{1}{32} \right)^6 > 1
\]
for every $x \in \TTT$. Since $f_n$ has constant
Jacobian, this imples that $f_n$ is collectively for large $n$.

\end{example}







\section{An applicability criterion in dimension two}

\begin{theorem} \label{thm:2D}
	Let $A \in \Mtt$ be a matrix of degree $J = |\det A| \geq 2$ with
	invariant factors $(\tau_1, \tau_2)$. 
	If
	\begin{equation} \label{eq:algebraic_cond}
 \|A \| - \m(A) < \frac{J^{\tau_2} - 1}{\tau_2 J^{(\tau_2-1)/2}}
	\end{equation}
then there exsits an ergodic conservative
real analytic endomorphism $f: \TT \to \TT$ with constant Jacobian, homotopic
to $A$, which is colectively expanding.
\end{theorem}

\begin{proof}[Proof of Theorem \ref{thm:2D}]
	Let $\hA=A/ \sqrt{J}$ so that $|\det \hA|=1$. 
By Lemma~\ref{lem:orthogonal_times_shear_dim2}, we have 
$\hA=QT$ for some $Q\in \Ort(2,\R)$ and $T\in\GT$. In particular, $\|T\|=\|\hA\|$, 
so $T$ can be written as $T = R \Sigma_t R^{-1}$ for $t = \psi(\|\hA\|)$ and
some $R \in \Ort(2,\R)$. 

By Lemma~\ref{lem:density_of_A_transvections}, we can find a sequence $S_k  \in \ST_A$
converging to $T^{-1}$. Let us write these as $P_k \Sigma_{t_k} P_k^{-1}$ for $P_k \in
\SL(2,\Z)$.

For $0<c<\tau_2^{-1}$ and $n\geq1$, consider
\[
	f_{k, c, n} =A\circ P_k\circ S_{c,n,t_k}\circ P_k^{-1}.
\]
Here, $S_{c,n,t}$ abbreviates $S_{1,2,c,n,t}$, so that
\[S_{c,n,t}(x,y) = (x + t s_{c,n}(y), y).\]
Each $f_{k,c,n}$ is a conservative real analytic endomorphism, homotopic to $A$, with
constant Jacobian $J$. 

We first obtain a global lower bound on the conorm of $Df_{k,c,n}$.
Writing $\alpha = s_{c,n}'(y)$,
we have
\begin{align*}
	D f_{k,c,n}(P_k(x,y))  & = A P_k \Sigma_{\alpha t_k} P_k^{-1} \\
			     & = A P_k \left[ I + \alpha (\Sigma_{t_k} - I) \right]
			     P_k^{-1} \\
			     & = A \left[ I + \alpha(S_k - I) \right]
\end{align*}
for every $(x,y) \in \TT$. But $A \left[ I + \alpha(S_k - I) \right] $
converges to 
\begin{align}
	A \left[ I + \alpha(T^{-1} - I) \right] 
	& =  A\left[ I + \alpha(R \Sigma_{-t} R^{-1} - I)\right] \\
	& =  A R \Sigma_{-\alpha t} R^{-1} \\
	& = \sqrt{J} Q R \Sigma_{t(1-\alpha)}R^{-1} \label{eq:bad_derivative}
\end{align}
uniformly for $\alpha$ in any compact set.
From Lemma~\ref{lem:sprime_convolution} we see that $1-\frac{1}{c} \leq \alpha
\leq 1$, so the conorm of \eqref{eq:bad_derivative} is bounded below 
by $\sqrt{J}/varphi(t/c)$. Thus for any $\delta>0$ and sufficiently large $k$, we have 
\[\m ( Df_{k,c,n}(x,y) ) > \frac{\sqrt{J}}{varphi(t/c)}-\delta \]
	for every $(x,y) \in \TT$ and every $n$. 
Let $L=J^{\tau_2/2}$. Since
$\m(A)=J/\|A\|$, condition~\eqref{eq:algebraic_cond}
is equivalent to
\[
\tau_2\psi(|\hA|)<\psi(L).
\]
As $t=\psi(|\hA|)$, it follows that
\[
varphi(t/c)<L
\]
for $c<\tau_2^{-1}$ sufficiently close to $\tau_2^{-1}$.
Fix such a $c$. Since
\[
\frac{\sqrt J}{varphi(t/c)}>\frac{\sqrt J}{L},
\]
we can choose $\eta>0$ and then $k$ sufficiently large so that
\[
\m(Df_{k,c,n}(x))>(1+\eta)\frac{\sqrt J}{L}
\]
for every $x\in\TT$ and every $n$.

We now show that the conorm of $Df_{k,c,n}$ is close to $\sqrt{J}$ for
points outside the "critical strip".
For $0< \rho <1$ write 
\[
	\Delta_\rho = \T \times \left(-\frac{\rho}{2} , \frac{\rho}{2}\right).
\]
Then, for every $c< \tau_2^{-1}$, the derivative of $P_k \circ
S_{c,n, t_k} P_k^{-1}$ converges to $S_k$ uniformly on 
the complement of $P_k \Delta_{\tau_2^{-1}}$ 
as $n \to \infty$. 
Fix $\epsilon>0$ sufficiently small that
\[
(1-\epsilon)^{J-\tau_1}(1+\eta)^{\tau_1}>1.
\]
Increasing $k$ if necessary, and then choosing $n$ sufficiently
large, we obtain
\[
\m(Df_{k,c,n}(x))>\sqrt J(1-\epsilon)
\]
for every $x\notin P_k\Delta_{\tau_2^{-1}}$.
The global lower bound established above remains valid.

Lemma~\ref{lem:preimages_in_critical_stripe_dimd} guarantees that
\[
\#\bigl(f_{k,c,n}^{-1}(x)\cap
P_k\Delta_{\tau_2^{-1}}\bigr)\leq\tau_1.
\]
Hence, using $L^{\tau_1} = J^{J/2}$, we have that
\begin{align}
	\prod_{y\in f_{k,c,n}^{-1}(x)}\m(Df_{k,c,n}(y))
& > \left(\sqrt{J} (1-\epsilon)\right)^{J-\tau_1}
\left(\frac{\sqrt{ J }}{L} (1+\eta) \right)^{\tau_1} \\
&= (1-\epsilon)^{J-\tau_1} (1+\eta)^{\tau_1} >1.
\end{align}
We have therefore shown that $k$, $c$, and $n$ can be chosen so that the map 
$f = f_{k,c,n}$ is collectively
exapnding. By Theorem~\ref{thm:ergodic_mexpanding}, $f$ is non-uniformly expanding and
stably ergodic. In particular, $f$ is ergodic and $C^1$ robustly
non-uniformly expanding.
\end{proof}

\bibliographystyle{alpha}
\bibliography{biblio}

\end{document}